\documentclass[12pt]{article}

\usepackage[top=1.2in,bottom=1.2in,left=1in,right=1in]{geometry}
\usepackage{amsmath,amssymb,amsfonts,amsthm,mathtools}
\usepackage{bm}
\usepackage{graphicx}
\usepackage{booktabs}
\usepackage{caption}
\usepackage{xcolor}
\usepackage{cite}
\usepackage[hidelinks]{hyperref}
\usepackage[capitalise,nameinlink]{cleveref}

\DeclareGraphicsExtensions{.pdf,.png,.jpg,.jpeg,.eps}

\newtheorem{theorem}{Theorem}[section]
\newtheorem{proposition}{Proposition}[section]
\newtheorem{lemma}{Lemma}
\newtheorem{definition}{Definition}

\theoremstyle{remark}
\newtheorem{remark}{Remark}
\numberwithin{figure}{section}
\numberwithin{table}{section}
\numberwithin{equation}{section}

\title{Wave Breaking and Structure-Preserving Fully Discrete Method for Stochastic Camassa--Holm Equation with Additive Noise}
\author{
Liying Sun$^{1}$, Zizhao Sun$^{2}$, and Liying Zhang$^{2}$\\
{\small $^{1}$Academy for Multidisciplinary Studies, Capital Normal University}\\
{\small $^{2}$School of Mathematical Science, China University of Mining and Technology (Beijing)}\\
{\small \texttt{liyingsun@lsec.cc.ac.cn}, \texttt{SQT2400702046@student.cumtb.edu.cn},}\\
{\small \texttt{lyzhang@lsec.cc.ac.cn}}
}
\date{}

\hypersetup{
  pdftitle={Wave Breaking and Structure-Preserving Fully Discrete Method for Stochastic Camassa--Holm Equation with Additive Noise},
  pdfauthor={Liying Sun, Zizhao Sun, and Liying Zhang}
}

\begin{document}

\maketitle

% REQUIRED
\begin{abstract}
In this article, we investigate  the  stochastic Camassa--Holm equation driven by additive noise, which models nonlinear shallow-water waves under random external forcing.  
We establish that, almost surely, any
finite-time breakdown occurs through wave breaking: the wave amplitude and the positive part of the slope remain bounded,
whereas the minimum slope tends to $-\infty$. 
Under an additional moment assumption, we derive an initial-slope
condition ensuring wave breaking with positive probability. We also
formulate the equation as a stochastic Hamiltonian PDE, establish its
stochastic multi-symplectic conservation law, and derive the evolution
law for the averaged $H^1$ energy. 
Combining the Fourier--Galerkin method, 
symplectic Runge--Kutta integration, and
exact solution of the linear stochastic subsystem, we propose a 
fully discrete method which
preserves an integrated discrete stochastic multi-symplectic conservation law and
the linear growth of the  discrete averaged  \(H^1\) energy. 
Under suitable regularity and resolution assumptions, we  present 
that the fully discrete dynamics retain the slope-steepening mechanism
responsible for wave breaking, providing a discrete counterpart of the
continuous behavior.

\par\medskip\noindent\textbf{AMS subject classification:} 60H15, 35B44, 65M70, 65P10.

\par\medskip\noindent\textbf{Key Words:}\\
Stochastic Camassa--Holm equation, additive noise, wave breaking,
stochastic multi-symplectic conservation law, structure-preserving fully discrete method.
\end{abstract}

\section{Introduction}
\label{sec:intro}
The Camassa--Holm (CH) equation is a nonlinear dispersive model for the
unidirectional propagation of long waves in shallow water
\cite{CamassaHolm1993,ConstantinLannes2009}.
Its integrable and Hamiltonian structures have motivated extensive
analytical study, including recent work on conservative
infinite-peakon solutions on the real line via inverse spectral methods
\cite{ChangEckhardtKostenko2025}.
The CH equation also admits solutions that undergo wave
breaking, in which the wave profile remains bounded, whereas its slope
becomes unbounded from below in finite time
\cite{ConstantinEscher1998}. This distinction makes the evolution of the spatial slope a central
issue in both the analysis and numerical approximation of the equation.

In physical wave environments, random external forcing and unresolved small-scale processes motivate stochastic
extensions of the deterministic CH equation. Here, we consider additive noise as an idealized representation
of random external forcing \cite{ChenGaoGuo2012}. 
Periodic boundary conditions provide an idealized setting for the
propagation of wave trains.
For the model considered here, additive noise replaces deterministic
$H^1$-energy conservation with an energy evolution law. 
The resulting averaged energy evolution law does not by itself characterize pathwise
singularity formation, which requires control of the amplitude and
separate analysis of the negative slope.
\iffalse
  Since the deterministic CH wave breaking is characterized by
an unbounded negative slope while the amplitude remains bounded, energy
estimates alone cannot determine whether finite-time breakdown occurs
through this bounded-amplitude wave-breaking mechanism.
The numerical approximation introduces a further difficulty.  A method
may be accurate while the solution remains smooth, yet still modify the
nonlinear instability that drives slope steepening.  In particular,
spatial projection and time discretization may alter the characteristic
slope dynamics near breaking.
\fi
Numerical approximation introduces an additional difficulty: accuracy
on smooth solution intervals does not ensure that the discretization
retains the instability responsible for steepening.  Spatial projection
and time discretization perturb the characteristic slope dynamics,
and their effect near breaking must therefore be controlled.
For the stochastic problem, a suitable
method should also reproduce the associated stochastic energy evolution induced by the additive noise.
Thus, accuracy in the smooth regime, consistency with stochastic energy
evolution, and faithful resolution of the nonlinear breaking mechanism
are related, but distinct, requirements.
Structure-preserving discretization provides a natural framework for
addressing these requirements.

For the deterministic CH equation,  considerable effort
has been devoted to numerical methods that inherit the intrinsic
geometric and conservative structures, including multi-symplectic,
Hamiltonian-conserving, and invariant-preserving discretizations
\cite{CohenOwrenRaynaud2008,Hong2019,LiuXing2016,Matsuo2010}.
Stochastic CH equations with  multiplicative noise
including transport- or convection-type perturbations, and pure-jump
noise have also been studied, with results on well-posedness,
qualitative dynamics, and wave breaking
\cite{AlbeverioBrzezniakDaletskii2021,ChenDuanGao2021,
ChenGaoGuo2012,CrisanHolm2018,Tang2018}.
Comparative numerical studies of stochastic Hamiltonian CH model
have also been carried out \cite{HolmSinghStreet2026}.
In the broader context of geometric numerical integration for
stochastic Hamiltonian PDEs, symplectic Runge--Kutta and
multi-symplectic methods have been developed
\cite{ChenHong2016,ChenHongJi2019,HongJiZhang2014,ZhangJi2019}.
These developments provide important insights into stochastic CH
dynamics and structure-preserving discretization. However, for the
CH equation with additive noise, a systematic understanding of its
intrinsic properties and their preservation under numerical
discretization is still lacking. This naturally raises the following
questions:
\begin{itemize}
    \item \noindent Does the wave-breaking mechanism survive in the case of the additive noise? 
    \item \noindent What are the geometric structure and energy evolution law of the stochastic system?
    \item \noindent How to design a fully discrete method that preserves the intrinsic properties of the original system?
\end{itemize}

The present work addresses these questions through the analysis of the
stochastic dynamics and their fully discrete method. First, we
characterize the finite-time breakdown mechanism of the stochastic CH
equation with additive noise.  Any finite-time loss of regularity is shown to be necessarily
of wave-breaking type: the wave amplitude and the positive part of the
spatial slope remain bounded up to the maximal existence time, whereas
the minimum slope diverges to $-\infty$. Under an additional moment
condition, this characterization yields a sufficient criterion for
finite-time wave breaking with positive probability in terms of a
sufficiently negative initial slope. 
Beyond singularity formation, we identify the geometric and energetic
structures underlying the stochastic dynamics. In particular, we derive
a stochastic multi-symplectic formulation and an evolution law for the
averaged $H^1$ energy, clarifying which structures of the deterministic
dynamics persist under additive stochastic forcing and providing the
basis for the numerical discretization. 
We then construct a fully discrete method combining a Fourier--Galerkin
approximation in space, a symplectic Runge--Kutta discretization of the
deterministic subsystem, and exact integration of the additive stochastic
subsystem. The resulting method satisfies an integrated discrete
multi-symplectic conservation law and preserves the linear-in-time
evolution of the discrete averaged $H^1$ energy. Thus, the principal
geometric and energetic structures of the stochastic CH equation are
retained at the fully discrete level. 
Preservation of these identities alone, however, does not guarantee that
the numerical solution captures the nonlinear mechanism of wave breaking.
We therefore analyze the numerical slope along discrete characteristics.
Under suitable regularity assumptions, its expectation
satisfies a Riccati-type inequality whose leading negative quadratic term
agrees with that of the continuous characteristic-slope equation up to
controlled discretization errors. Consequently, sufficiently negative
discrete slopes are incompatible with a uniform high-regularity bound
beyond the corresponding Riccati comparison time. 
This analysis establishes a quantitative correspondence between the
continuous and discrete wave-breaking mechanisms. 
We prove that the
proposed discretization preserves not only the geometric and energetic
structures, but also the nonlinear steepening mechanism responsible for
wave breaking. %, without requiring convergence of the numerical breaking times or direct approximation of the random singularity time. 
Finally, numerical experiments verify temporal convergence order of the proposed numerical method, reproduce the linear of the averaged $H^1$ energy,
and illustrate the progressive formation of steep descending fronts under
additive stochastic forcing.

The remainder of the paper is organized as follows. Section~2 studies the
continuous stochastic CH equation, including its stochastic
multi-symplectic formulation, averaged energy evolution, and wave-breaking
mechanism. Section~3 develops the fully discrete Fourier--Galerkin
splitting Runge--Kutta method and analyzes its structural properties and
discrete slope dynamics. Section~4 presents the numerical experiments.

\section[Stochastic Camassa--Holm equation with additive noise]{Stochastic Camassa--Holm equation\\with additive noise}
\label{sec.2}
We consider the following stochastic Camassa--Holm equation with additive noise
\begin{align}
d(u-u_{xx})
+
\bigl(
3uu_x-2u_xu_{xx}-uu_{xxx}
\bigr)\,dt
&=
\Phi\,dW(t),
\qquad
u(0)=u_0,
\label{eq:SCH_original}
\end{align}
where $W(\cdot)$ is a cylindrical Wiener process on the separable Hilbert
space $\mathcal U=L_0^2(\mathbb T_L) := \left\{ f \in L^2(\mathbb{T}_L) : \int_{\mathbb{T}_L} f(x) \, dx = 0 \right\}$. 
Here,
$\mathbb T_L:=\mathbb R/L\mathbb Z$ is identified with the periodic
interval $[a,b)$, where $L=b-a$. 
Throughout the paper, all spatial differential operators are understood in the periodic sense.
We set $\mu:=2\pi/L$ and
$\Lambda:=(1-\partial_x^2)^{-1}$. 
For $f\in L^2(\mathbb T_L)$, its Fourier coefficients are defined by $\widehat f(k) = \frac{1}{L} \int_a^b f(x)e^{-\mathrm{i}k\mu(x-a)}\,dx, k\in\mathbb Z,$  so that $f(x) = \sum\limits_{k\in\mathbb Z} \widehat f(k)e^{\mathrm{i}k\mu(x-a)} $ in the usual $L^2$ sense.
We equip $H^r(\mathbb T_L)$ with the norm $\|f\|_{H^r}^2 = \sum\limits_{k\in\mathbb Z} (1+|k\mu|^2)^r |\widehat f(k)|^2. $
The corresponding zero-mean subspace is denoted by
$H_0^r(\mathbb T_L)
=
\{f\in H^r(\mathbb T_L):\widehat f(0)=0\}$.
For $f\in H_0^r(\mathbb T_L)$, we define the zero-mean antiderivative
$\partial_x^{-1}f$ through
$\widehat{\partial_x^{-1}f}(k)
=
(\mathrm i k\mu)^{-1}\widehat f(k)$ for $k\ne0$, and
$\widehat{\partial_x^{-1}f}(0)=0$.
Then
$\partial_x^{-1}:H_0^r(\mathbb T_L)\to
H_0^{r+1}(\mathbb T_L)$
is bounded and is the inverse of $\partial_x$ on the zero-mean
subspace.

Applying $\Lambda$ to \eqref{eq:SCH_original} and using the identity
\begin{equation*}
\Lambda\left(3uu_x-2u_xu_{xx}-uu_{xxx}\right)
=uu_x+\Lambda\partial_x\left(u^2+\frac12u_x^2\right),
\end{equation*}
 we obtain the equivalent  formulation
\begin{align}
du+uu_x\,dt
&=
-\Lambda\partial_x
\left(
u^2+\frac12u_x^2
\right)\,dt
+
\Lambda\Phi\,dW(t),
\label{eq:sch-nonlocal}
\end{align}
which is  convenient for the subsequent analysis, since the nonlinear drift contains only first derivatives of $u$ together with the smoothing operator $\Lambda$.

We first recall the well-posedness result for the stochastic Camassa--Holm equation. The local well-posedness and continuation properties of
\eqref{eq:sch-nonlocal} follow from
\cite[Theorem~2.1]{RohdeTang2021}. Indeed, for $k=1$, the nonlocal
drift considered in \cite{RohdeTang2021} coincides with that in
\eqref{eq:sch-nonlocal}, while the additive noise considered here
corresponds to the state-independent coefficient
$h(t,u)\equiv\Lambda\Phi$. Hence the assumptions on the stochastic
coefficient in \cite{RohdeTang2021} are satisfied in the present
setting.

\begin{lemma}
\label{prop:continuous_framework}
Let $s>3/2$, let $u_0\in H^s(\mathbb T_L)$ be deterministic, and assume
\begin{equation*}
\Lambda\Phi\in\mathcal L_2(\mathcal U;H^s(\mathbb T_L)).
\end{equation*}
Then there exists a unique maximal pathwise solution
$(u,\tau_\infty)$ of \eqref{eq:sch-nonlocal} such that
\begin{equation*}
u\in C([0,\tau_\infty);H^s(\mathbb T_L))\quad\text{almost surely}.
\end{equation*}
Moreover,
\begin{align}
\mathbf 1_{\{\tau_\infty<\infty\}}
&=
\mathbf 1_{\{
\limsup\limits_{t\uparrow\tau_\infty}
\|u(t)\|_{W^{1,\infty}(\mathbb T_L)}
=\infty\}}
\qquad\text{a.s.}
\label{eq:continuous_continuation}
\end{align}
\end{lemma}

For the subsequent analysis  in Section \ref{sec.2}, we assume
$s>\frac52$, $u_0\in H_0^s(\mathbb T_L)$, and
$\Phi\in\mathcal L_2(\mathcal U;H_0^{s-2}(\mathbb T_L))$.
% The condition $s>\frac52$ ensures the continuous embedding
% $H^s(\mathbb T_L)\hookrightarrow W^{1,\infty}(\mathbb T_L),$
% which is sufficient for the continuous and spatially semi-discrete
% slope estimates considered below.
\iffalse
Since $\Lambda$ is a Fourier multiplier of order $-2$, $\Lambda: H_0^{s-2}(\mathbb T_L) \longrightarrow H_0^s(\mathbb T_L)$ is bounded, and hence $\Lambda\Phi \in \mathcal L_2 \bigl( \mathcal U;H_0^s(\mathbb T_L) \bigr).$

 The restriction $s>5/2$ guarantees the continuous embedding $H^s(\mathbb T_L) \hookrightarrow W^{1,\infty}(\mathbb T_L),$ which will be used repeatedly in the subsequent regularity and slope estimates.
For the analysis below, we assume that $s>5/2$,
$u_0\in H_0^s(\mathbb T_L)$ is deterministic, and
$\Phi\in
\mathcal L_2(\mathcal U;H_0^{s-2}(\mathbb T_L))$.
Since $\Lambda$ is a Fourier multiplier of order $-2$, it follows that
$\Lambda\Phi\in
\mathcal L_2(\mathcal U;H_0^s(\mathbb T_L))$. 
\fi
% Moreover, the zero-mean subspace is invariant under the dynamics. 
% Indeed, since
% $u_0\in H_0^s(\mathbb T_L)$ and the range of $\Phi$ is contained in
% $H_0^{s-2}(\mathbb T_L)$, 
Under this assumption, 
the stochastic forcing in
\eqref{eq:sch-nonlocal} has zero spatial mean. Integrating
\eqref{eq:sch-nonlocal} over $\mathbb T_L$ and using periodicity gives
$d\int_{\mathbb T_L}u(t,x)\,dx
=0.$
Consequently,
$\int_{\mathbb T_L}u(t,x)\,dx
=
\int_{\mathbb T_L}u_0(x)\,dx
=0$ for
$0\le t<\tau_\infty$ a.s.
Therefore, $u(t)\in H_0^s(\mathbb T_L)$ throughout its maximal interval
of existence, and the operator $\partial_x^{-1}$ is well defined along
the solution.

\subsection{Stochastic multi-symplectic structure}
Since $\Lambda$ and $\partial_x\Lambda$ are Fourier multipliers of
orders $-2$ and $-1$, respectively, we have
$G_0:=\Lambda\Phi\in\mathcal L_2(\mathcal U;H_0^s(\mathbb T_L))$ and
$\partial_xG_0\in
\mathcal L_2(\mathcal U;H_0^{s-1}(\mathbb T_L))$. Noting that $\partial_x^{-1}:H_0^{s-2}(\mathbb T_L) \to H_0^{s-1}(\mathbb T_L)$ is bounded, we have $\partial_x^{-1}\circ\Phi \in \mathcal L_2 \bigl( \mathcal U;H_0^{s-1}(\mathbb T_L) \bigr).$ We therefore introduce the auxiliary $H_0^{s-1}(\mathbb T_L)$-valued Wiener process $B(t) := \int_0^t \partial_x^{-1}\Phi\,dW(r), B(0)=0,$ whose covariance operator is $Q_B = (\partial_x^{-1}\Phi) (\partial_x^{-1}\Phi)^*. $
For convenience, set
$H:=u^2+\frac12u_x^2$ and
$F(u):=-uu_x-\Lambda\partial_xH$.
\iffalse
{\color{blue}
It can be verified that
there exists $C>0$ such that, for all $v,w\in H^3(\mathbb T_L)$,
with $e=v-w$,
\begin{align}
\left\langle e,F(v)-F(w)\right\rangle_{H^2}
&\le
C\bigl(\|v\|_{H^3}+\|w\|_{H^3}\bigr)
\|e\|_{H^2}^2,
\label{eq:CH_H2_one_sided}\\
\|F(v)-F(w)\|_{H^2}
&\le
C\bigl(\|v\|_{H^3}+\|w\|_{H^3}\bigr)
\|v-w\|_{H^3}.
\label{eq:CH_H3_to_H2_lipschitz}
\end{align}
}
\begin{proof}
Writing $e=v-w$, we have
\begin{align*}
F(v)-F(w)
={}&
-v\,\partial_xe
-e\,\partial_xw
\\
&-
\partial_x\Lambda
\left[
(v+w)e
+
\frac12(\partial_xv+\partial_xw)\partial_xe
\right].
\end{align*}
For the transport part, differentiating up to order two and integrating
the highest-order term by parts gives
\begin{align*}
\left|
\left\langle
e,
v\,\partial_xe+e\,\partial_xw
\right\rangle_{H^2}
\right|
\le
C\bigl(\|v\|_{H^3}+\|w\|_{H^3}\bigr)
\|e\|_{H^2}^2.
\end{align*}
Moreover, $\partial_x\Lambda$ is a Fourier multiplier of order $-1$.
Since $H^1(\mathbb T_L)$ is an algebra,
\begin{align*}
&\left\|
\partial_x\Lambda
\left[
(v+w)e
+
\frac12(\partial_xv+\partial_xw)\partial_xe
\right]
\right\|_{H^2}
\\
&\qquad\le
C\bigl(\|v\|_{H^3}+\|w\|_{H^3}\bigr)
\|e\|_{H^2}.
\end{align*}
This proves \eqref{eq:CH_H2_one_sided}.  The product estimate
$\|fg\|_{H^2}\le
C(\|f\|_{H^2}\|g\|_{L^\infty}
+\|f\|_{L^\infty}\|g\|_{H^2})$,
together with $H^2(\mathbb T_L)\hookrightarrow W^{1,\infty}(\mathbb T_L)$,
gives \eqref{eq:CH_H3_to_H2_lipschitz}.
\end{proof}
\fi
Using
$(1-\partial_x^2)(uu_x)+\partial_xH
=
3uu_x-2u_xu_{xx}-uu_{xxx}$,
equation~\eqref{eq:sch-nonlocal} can be written as
\begin{equation} \label{eq:abstract_SCH} du=F(u)\,dt+G_0\,dW(t). \end{equation}
We further introduce the auxiliary variables $p:=u_x,  w:=-\frac14u^2-\frac12\Lambda H(u), v:=up+\frac12F(u).$ These variables allow us to rewrite the stochastic CH equation as an stochastic  Hamiltonian partial differential equation, as stated in Proposition \ref{prop:MS_formulation}. 

\begin{proposition}
\label{prop:MS_formulation}
Assume that 
$u_0\in H_0^s(\mathbb T_L)$ 
and
$\Phi\in
\mathcal L_2
(\mathcal U;H_0^{s-2}(\mathbb T_L)), 
$ where $s>5/2$. 
%Let $u$ be a sufficiently regular  solution of \eqref{eq:SCH_original} satisfying $u\in C([0,\tau);H_0^s(\mathbb T_L))$, a.s., 
%and 
Let $\phi$ be defined by
$d\phi=2w\,dt+\partial_x^{-1}G_0 dW(t)$, $\phi(0)=\phi_0$,
where $\phi_0$ satisfies
$(\phi_0)_x=u_0.$
Then \eqref{eq:SCH_original} is equivalent to
\begin{align}
\left\{
\begin{aligned}
&\frac12d\phi-\frac12dp-v_xdt
=
\left(
-w-\frac32u^2-\frac12p^2
\right)dt
+\frac12dB(t),
\\
&-\frac12du+w_xdt
=
-\frac12G_0 dW(t),
\quad -\phi_xdt
=
-u\,dt,
\quad u_xdt
=
p\,dt,
\\
&\frac12du
=
(-up+v)dt+\frac12G_0  dW(t).
\end{aligned}
\right.
\label{eq:first_order_system}
\end{align}
Equivalently,
\begin{align}
M\,d\mathbf z+K\mathbf z_xdt
=
\nabla_{\mathbf z}S_1(\mathbf z)dt
+\mathcal G dW(t).
\label{eq:multi_symplectic_form}
\end{align}
Here $\mathbf z=[u,\phi,w,v,p]^{\mathrm T}$,
$S_1(\mathbf z)
=
-wu-\frac12u^3-\frac12up^2+pv,
$
and
\begin{align*}
M&=
\begin{bmatrix}
0&\frac12&0&0&-\frac12\\
-\frac12&0&0&0&0\\
0&0&0&0&0\\
0&0&0&0&0\\
\frac12&0&0&0&0
\end{bmatrix},
\quad 
K=
\begin{bmatrix}
0&0&0&-1&0\\
0&0&1&0&0\\
0&-1&0&0&0\\
1&0&0&0&0\\
0&0&0&0&0
\end{bmatrix}, 
\quad
\mathcal Gh
=
\frac12
\begin{bmatrix}
\partial_x^{-1}\Phi h\\
-\Lambda\Phi h\\
0\\
0\\
\Lambda\Phi h
\end{bmatrix}.
\end{align*}
\end{proposition}

\begin{proof}
By the definitions of $w$ and $v$, we have
$2w_x
=-\frac12\partial_x(u^2)-\partial_x\Lambda H
=
-uu_x-\Lambda\partial_xH
=F(u)$ and $v=up+\frac 12F(u)$. 
Next, taking differentiation of $\phi$ with respect to $x$ gives
$d\phi_x
=
2w_x\,dt
+
G_0\,dW(t)=F(u)\,dt+G_0\,dW(t).$
Combining with \eqref{eq:abstract_SCH} yields
$d(\phi_x-u)=0.$
Moreover, differentiating \eqref{eq:abstract_SCH} with respect
to $x$ gives $du_x
=
F_x\,dt+\partial_xG_0\,dW(t).$
Hence, using $p=u_x$ and $v=up+\frac 12F(u)$, we obtain
\begin{align*}
\frac12d\phi-\frac12dp-v_x\,dt
={}&
\left(
w-F_x-(up)_x
\right)dt
+\frac12
\left(
\partial_x^{-1}G_0-\partial_xG_0
\right)dW(t).
\end{align*}
Observe that $F(u)+up
=-\Lambda\partial_xH,$
which implies $F_x+(up)_x
=
-\Lambda\partial_x^2H.$
Since $\Lambda=(1-\partial_x^2)^{-1},$
we deduce $-\Lambda\partial_x^2H
=(I-\Lambda)H,$ and thus $F_x+(up)_x=(I-\Lambda)H.$
Making use of $G_0=\Lambda\Phi$, we obtain
$\partial_x^{-1}G_0-\partial_xG_0
=
\partial_x^{-1}\Lambda\Phi
-\partial_x\Lambda\Phi
=\partial_x^{-1}\Phi,$
which yields the first equation in
\eqref{eq:first_order_system}.
\iffalse
Moreover, $p=u_x$ and $d(\phi_x-u)=0$.
The first component follows from
$F_x+(up)_x=(I-\Lambda)H$ and
$\partial_x^{-1}G_0-\partial_xG_0=\partial_x^{-1}\Phi$.
\fi
The remaining four equations in
\eqref{eq:first_order_system} follow directly from 
substituting these identities to  obtain
\eqref{eq:first_order_system} and hence
\eqref{eq:multi_symplectic_form}.
Conversely, eliminating the auxiliary variables recovers
\eqref{eq:SCH_original}.
\end{proof}

Now we investigate the stochastic multi-symplectic conservation law of \eqref{eq:multi_symplectic_form}.

\begin{proposition}
Assume that the stochastic multi-symplectic system
\eqref{eq:multi_symplectic_form} admits a solution
$\mathbf z=[u,\phi,w,v,p]^{\mathrm T}$ such that
\begin{itemize}
\item $\mathbf z$ is sufficiently regular in the spatial variable and
$\mathbf z_x$ is well defined;
\item the phase-space variation
$\xi=\mathbf d\mathbf z$ exists and is sufficiently regular.
\end{itemize}
Then the local stochastic multi-symplectic conservation law
\begin{align}
d\omega+\partial_x\kappa\,dt=0,\quad a.s.,
\label{eq:MS_conservation}
\end{align}
holds, where
\begin{align*}
\omega
&=
\frac12\mathbf d\mathbf z\wedge M\mathbf d\mathbf z
=
\frac12\mathbf du\wedge\mathbf d\phi
-\frac12\mathbf du\wedge\mathbf dp,
\; 
\kappa
=
\frac12\mathbf d\mathbf z\wedge K\mathbf d\mathbf z
=
-\mathbf du\wedge\mathbf dv
+\mathbf d\phi\wedge\mathbf dw .
\end{align*}
Here $\mathbf d$ denotes the exterior derivative with respect to the
phase variables, whereas $d$ denotes the stochastic time differential.
\end{proposition}

\begin{proof}
Let $\xi=\mathbf d\mathbf z$. 
Taking the exterior derivative of
\eqref{eq:multi_symplectic_form} %, and using the fact that $M$ and $K$ are constant matrices, gives 
implies 
%the variational equation
\begin{equation}
\label{eq:variational_MS}
M\,d\xi+K\xi_x\,dt
=
S_1''(\mathbf z)\xi\,dt.
\end{equation}
% Since the noise operator $\mathcal G$ is independent of the phase
% variables, its phase space exterior derivative vanishes. Thus, the
% additive stochastic forcing does not contribute to the variational
% equation \eqref{eq:variational_MS}.
Taking the wedge product of \eqref{eq:variational_MS} with $\xi$ 
together with the skew symmetry of $M$ and $K$, we obtain
$d\bigl(\frac12\xi\wedge M\xi\bigr)+
\partial_x\bigl(\frac12\xi\wedge K\xi\bigr)dt=0$,
which is precisely \eqref{eq:MS_conservation}.
% {\color{red} For the present matrix $M$ and $K$, a direct calculation together with the antisymmetry of the wedge product, we can obtain the expression for $\omega$ and $\kappa.$}
\iffalse
$\frac12\mathbf d\mathbf z\wedge M\mathbf d\mathbf z
=
\frac12\left[
\frac12\mathbf du\wedge\mathbf d\phi
-\frac12\mathbf du\wedge\mathbf dp
-\frac12\mathbf d\phi\wedge\mathbf du
+\frac12\mathbf dp\wedge\mathbf du
\right]
=
\frac12\mathbf du\wedge\mathbf d\phi
-\frac12\mathbf du\wedge\mathbf dp,$
where the antisymmetry of the wedge product has been used. Thus, $\omega
=
\frac12\mathbf du\wedge\mathbf d\phi
-\frac12\mathbf du\wedge\mathbf dp.$
A similar calculation yields the expression for $\kappa$.
\fi
\end{proof}

We next investigate the averaged energy evolution law of the original equation. Together with the multi-symplectic formulation above, this gives another intrinsic  property.

\subsection{Averaged energy evolution law and characteristic  slope law}
Define the energy functional of \eqref{eq:multi_symplectic_form} by
$E(t):=\frac12\|u(t)\|_{H^1}^2$.
Set
\begin{equation*}
F(u):=-uu_x-\Lambda\partial_x\left(u^2+\frac12u_x^2\right).
\end{equation*}
Then the equation can be written as
$du=F(u)\,dt+G_0\,dW(t).$ 
For $R\ge1$, we introduce the stopping times
\begin{equation*}
\tau_R:=\inf\left\{t\in[0,\tau_\infty):
\|u(t)\|_{H^s}\ge R\right\}\wedge R,
\end{equation*}
with the convention $\inf\emptyset=\infty$.
% {\color{red} It is immediate that $\{\tau_R\}_{R\ge1}$ is nondecreasing in $R$. Thus, $\tau_R\le R$ almost surely.} 
%Moreover, s
Since $u\in C([0,\tau_\infty);H^s)$ almost surely and
$H^s(\mathbb T_L)\hookrightarrow W^{1,\infty}(\mathbb T_L)$ for $s>5/2$,
the continuation criterion \eqref{eq:continuous_continuation} implies
$\tau_R\uparrow\tau_\infty$ a.s. as $R\to\infty$. 
Applying It\^o's formula to the stopped process $u(t\wedge\tau_R)$, we obtain
\begin{align}
\|u(t\wedge\tau_R)\|_{H^1}^2
&=
\|u_0\|_{H^1}^2
+2\int_0^{t}{\bf{1}}_{[0,\tau_R]}(s)(u(s),\Lambda\Phi\,dW(s))_{H^1}
+C_\Phi(t\wedge\tau_R),
\label{eq:stopped_H1_energy}
\end{align}
where 
$C_\Phi
:=
\|\Lambda\Phi\|_{\mathcal L_2(\mathcal U;H^1(\mathbb T_L))}^2
=
\|\Lambda\Phi\|_{\mathcal L_2(\mathcal U;L^2(\mathbb T_L))}^2
+
\|\partial_x\Lambda\Phi\|_{\mathcal L_2(\mathcal U;L^2(\mathbb T_L))}^2.$
%\label{eq:energy_growth_constant}
% {\color{red}The stopped stochastic integral is a square-integrable martingale with zero expectation.} 
Taking expectation  %therefore gives
leads to
\begin{align}
\mathbb E\|u(t\wedge\tau_R)\|_{H^1}^2
=
\|u_0\|_{H^1}^2
+
C_\Phi\mathbb E[t\wedge\tau_R].
\label{eq:energy_growth_h1}
\end{align} 
Suppose that 
$\mathbb P(\tau_\infty>T)=1
$ 
and 
$\mathbb E\sup\limits_{0\le t\le T}
\|u(t)\|_{H^s}^2<\infty.$ 
Then $\left\{
\|u(t\wedge\tau_R)\|_{H^1}^2
\right\}_{R\ge1}$
is uniformly integrable for each $t\in[0,T]$.   Letting
$R\to\infty$ in \eqref{eq:energy_growth_h1}, we obtain, for every
$0\le t\le T$,
\begin{align*}
\mathbb E[E(t)]
=
\frac12\|u_0\|_{H^1}^2
+\frac12C_\Phi t.
\end{align*}

We next relate the energy estimate to the slope dynamics by deriving an evolution equation for the slope along characteristics. 
For each $x\in\mathbb R$, let $q(t,x)$ be the characteristic flow
defined by $\frac{d}{dt}q(t,x)=u(t,q(t,x)), q(0,x)=x.$ 
As shown in \cite{ChenDuanGao2021},  taking integration implies 
$q(t,x)=x+\int_0^t u(s,q(s,x))\,ds$, and thus 
$\partial_xq(t,x)=1+\int_0^t
u_x(s,q(s,x))\partial_xq(s,x)ds$.
Consequently, 
$\frac{d}{dt}\partial_xq(t,x)
=u_x(t,q(t,x))\partial_xq(t,x)$,
with $\partial_xq(0,x)=1$.
Solving this scalar equation yields \begin{align}
\label{partialx:q}
\partial_xq(t,x)
=
\exp\left(
\int_0^t u_x(s,q(s,x))\,ds
\right)>0.
\end{align} 
Making use of $\partial_xq(t,x)>0$, the periodicity of $u$ and uniqueness of the characteristic
equation, we deduce $q(t,x+L)=q(t,x)+L.$
Hence, $q(t,\cdot)$ induces an orientation-preserving
$C^1$-diffeomorphism of $\mathbb T_L$.  

% By the periodicity of $u$, the flow satisfies
% $q(t,x+L)=q(t,x)+L$ and therefore induces a flow on
% $\mathbb T_L$.

% \begin{lemma}
% Assume that $s>5/2$ and
% $u\in C([0,\tau_\infty(\omega));H^s(\mathbb T_L))$.
% Then, for almost every $\omega$, the characteristic flow
% \begin{align}
% \label{partialx:q}
% \partial_xq(t,x)
% =
% \exp\left(
% \int_0^t u_x(s,q(s,x))\,ds
% \right)>0.
% \end{align}
% Moreover, for every $t<\tau_\infty$, the map
% $q(t,\cdot):\mathbb T_L\to\mathbb T_L$ is an orientation-preserving $C^1$-diffeomorphism a.s. 
% \end{lemma}

% \begin{proof}
% Taking integration, we obtain
% $q(t,x)=x+\int_0^t u(s,q(s,x))\,ds$ which yields 
% $\partial_xq(t,x)=1+\int_0^t
% u_x(s,q(s,x))\partial_xq(s,x)ds$.
% Thus, 
% $\frac{d}{dt}\partial_xq(t,x)
% =u_x(t,q(t,x))\partial_xq(t,x)$,
% with $\partial_xq(0,x)=1$.
% Solving this scalar equation yields \eqref{partialx:q}. 
% Since $\partial_xq(t,x)>0$, the flow is strictly increasing.
% Moreover, periodicity of $u$ and uniqueness of the characteristic
% equation imply $q(t,x+L)=q(t,x)+L.$
% Hence $q(t,\cdot)$ induces an orientation-preserving
% $C^1$-diffeomorphism of $\mathbb T_L$. 
% \end{proof}

Let $v=u_x$. Differentiating 
% {\color{red}the nonlocal form of}   
\eqref{eq:sch-nonlocal}
with respect to $x$, we obtain
\begin{align*}
dv+\partial_x(uv)\,dt
=
-\partial_x^2\Lambda
\left(u^2+\frac12v^2\right)dt
+
\partial_x\Lambda\Phi\,dW(t).
\end{align*}
Using
$\partial_x(uv)=v^2+uv_x$ and
$-\partial_x^2\Lambda=I-\Lambda$, 
\begin{align*}
dv
=
\left[
-v^2-uv_x
+
(I-\Lambda)
\left(u^2+\frac12v^2\right)
\right]dt
+
\partial_x\Lambda\Phi\,dW(t).
\end{align*} 
% {\color{red}Since $s>5/2$, we have
% $H^s(\mathbb T_L)\hookrightarrow C^2(\mathbb T_L)$,
% so that $v=u_x$ and $v_x=u_{xx}$ admit continuous
% representatives.} 
Fix $x_0\in\mathbb T_L$ and define the slope along the characteristic
through $x_0$ by
$U(t):=v(t,q(t,x_0))$.
%We first verify that the stochastic term evaluated along the characteristic is well defined. 
Since $\partial_x\Lambda\Phi\in
\mathcal L_2(\mathcal U;H_0^{s-1}(\mathbb T_L))$, the Sobolev embedding
$H^{s-1}(\mathbb T_L)\hookrightarrow L^\infty(\mathbb T_L)$ gives
\begin{align*}
\sum_{j\ge1}\|(\partial_x\Lambda\Phi)e_j\|_{L^\infty}^2
&\le C\|\partial_x\Lambda\Phi\|_{\mathcal L_2(\mathcal U;H^{s-1})}^2
<\infty.
\end{align*}
Consequently,
\begin{equation*}
\sum_{j\ge1}(\partial_x\Lambda\Phi e_j)(q(t,x_0))\,d\beta_j(t)
\end{equation*}
defines a square-integrable martingale on every localized time
interval. 
As the characteristic flow has finite variation, %the chain rule for $v(t,q(t,x_0))$ indicates
we deduce
\begin{align*}
&dU(t)\\
=&
\big[
-v(t,q(t,x_0))^2
-u(t,q(t,x_0))v_x(t,q(t,x_0))
+
(I-\Lambda)
\big(u^2+\frac12v^2\big)
(t,q(t,x_0))
\big]dt
\\
&+
u(t,q(t,x_0))v_x(t,q(t,x_0))dt+
(\partial_x\Lambda\Phi\,dW(t))(q(t,x_0))\\
=&
\left[
-U(t)^2
+
(I-\Lambda)
\left(u^2+\frac12v^2\right)
(t,q(t,x_0))
\right]dt
+
(\partial_x\Lambda\Phi\,dW(t))(q(t,x_0)).
\end{align*}
Using
$(I-\Lambda)f=f-\Lambda f$ and $v=u_x$, we obtain
\begin{align}
\label{eq:dU_characteristic}
dU(t)=&
\left[
-\frac12U(t)^2
+
u(t,q(t,x_0))^2
-
\Lambda
\left(
u^2+\frac12u_x^2
\right)
(t,q(t,x_0))
\right]dt\\
&+
(\partial_x\Lambda\Phi\,dW(t))(q(t,x_0)).\notag
\end{align}
Integrating
\eqref{eq:dU_characteristic}, we obtain
\begin{align*}
U(t\wedge\tau_R)
&=
U(0)
+
\int_0^{t\wedge\tau_R}
\left[
-\frac12U(s)^2
+
u(s,q(s,x_0))^2
-
A(s)
\right]ds
+
M_{t\wedge\tau_R},
\end{align*}
where
$M_{t\wedge\tau_R}
:=
\int_0^{t\wedge\tau_R}
(\partial_x\Lambda\Phi\,dW(s))(q(s,x_0))$, 
and 
$A(t)
:=
\Lambda
\left(
u^2+\frac12u_x^2
\right)
(t,q(t,x_0)).$

% The preceding identities are first understood on the localized
% interval $[0,T\wedge\tau_R]$, where the characteristic flow and the
% pointwise evaluations are well defined. 

Let $G_L$ denote the periodic Green kernel of
$\Lambda=(1-\partial_x^2)^{-1}$. Since $G_L$ is strictly positive,
% $\Lambda$ is positivity preserving. Therefore, 
we obtain 
\begin{align*}
0\leq A(t)
&=\int_{\mathbb T_L}G_L(q(t,x_0)-y)
\left(u(t,y)^2+\frac12u_x(t,y)^2\right)\,dy\\
&\leq\|G_L\|_{L^\infty}
\left(\|u(t)\|_{L^2}^2+\frac12\|u_x(t)\|_{L^2}^2\right)\\
&\leq C_L\|u(t)\|_{H^1}^2.
\end{align*}
Since 
% $|U(t)|^2
% +
% |u(t,q(t,x_0))|^2
% \le
% C\|u(t)\|_{H^s}^2,$   
\iffalse
Moreover, denoting $G_L$ the periodic Green kernel of
$\Lambda=(1-\partial_x^2)^{-1}$, then
\begin{align*}
0\le %A(t)
%&=
\Lambda\left(u^2+\frac12u_x^2\right)
(t,q(t,x_0))
\le
\|G_L\|_{L^\infty}
\left(
\|u(t)\|_{L^2}^2
+
\frac12\|u_x(t)\|_{L^2}^2
\right)
\le
C_L\|u(t)\|_{H^1}^2.
\end{align*}
\fi
%and thus
$|U(t)|^2
+
|u(t,q(t,x_0))|^2
+
|A(t)|
\le
C_L\|u(t)\|_{H^s}^2,$
we obtain the right-hand side is integrable on
$\Omega\times[0,T]$ by means of
$\mathbb P(T<\tau_\infty)=1$ and
$\mathbb E\sup\limits_{0\le t\le T}\|u(t)\|_{H^s}^2<\infty$. 
It follows that the localization can be removed, the martingale term
has zero expectation, and $t\mapsto\mathbb E[U(t)]$ is absolutely
continuous. Consequently, for almost every $t\in[0,T]$,
\begin{align*}
\frac{d}{dt}\mathbb E[U(t)]
=
-\frac12\mathbb E[U(t)^2]
+
\mathbb E[u(t,q(t,x_0))^2]
-
\mathbb E[A(t)] .
\end{align*}

We next combine the energy estimate with the characteristic slope equation to characterize the wave breaking property of the original system.

\subsection{Wave breaking criterion}
\label{subsec:continuous_wave_breaking}
Let  
$ Z(t):=\int_0^t\Lambda\Phi\,dW(r),  t\geq0.$
Because $\Lambda\Phi\in
\mathcal L_2(\mathcal U;H_0^s(\mathbb T_L)),$
the stochastic convolution $Z$ has $H^s$-valued continuous
trajectories.
\iffalse
Since
$\Lambda\Phi
\in
\mathcal L_2(\mathcal U;H_0^s(\mathbb T_L))$,
the process $Z$ has continuous trajectories in
$H_0^s(\mathbb T_L)$. 
\fi
Moreover, since $s>5/2$,
\begin{align}
\sup_{0\le t\le T}
\left(
\|Z_x(t)\|_{L^\infty}
+
\|Z_{xx}(t)\|_{L^\infty}
\right)
&<\infty
\qquad\text{a.s.}
\label{eq:pathwise_Z_bound}
\end{align}
Now we study that any finite-time loss of regularity occurs through the unbounded steepening of the negative slope, while the solution amplitude and the positive part of the slope remain bounded.

\begin{proposition}
\label{prop:pathwise_wave_breaking}
Let $s>5/2$, let
$u_0\in H_0^s(\mathbb T_L)$ be deterministic, and assume
\begin{equation*}
\Phi\in\mathcal L_2(\mathcal U;H_0^{s-2}(\mathbb T_L)).
\end{equation*}
Let $(u,\tau_\infty)$ be the corresponding maximal pathwise solution.
Then there exists an event $\Omega_{\rm wb}$ with
$\mathbb P(\Omega_{\rm wb})=1$ such that, for every
$\omega\in\Omega_{\rm wb}$ satisfying
$\tau_\infty(\omega)<\infty$,
\begin{align}
\sup_{0\le t<\tau_\infty(\omega)}
\|u(t,\omega)\|_{L^\infty}
&<\infty,
\label{eq:pathwise_u_bound}\\
\sup_{0\le t<\tau_\infty(\omega)}
\sup_{x\in\mathbb T_L}
u_x(t,x,\omega)
&<\infty,
\label{eq:pathwise_positive_slope}\\
\lim_{t\uparrow\tau_\infty(\omega)}
\inf_{x\in\mathbb T_L}
u_x(t,x,\omega)
&=-\infty.
\label{eq:pathwise_negative_slope}
\end{align}
\end{proposition}

\begin{proof}
Define $M_R(t)
:=
2\int_0^{t\wedge\tau_R}
\bigl(u(\rho),\Lambda\Phi\,dW(\rho)\bigr)_{H^1}.$ 
Its quadratic variation satisfies
\begin{align*}
[M_R]_T=
4\int_0^{T\wedge\tau_R}
\|(\Lambda\Phi)^*u(\rho)\|_{\mathcal U}^2\,d\rho\leq
4C_\Phi
\int_0^{T\wedge\tau_R}
\|u(\rho)\|_{H^1}^2\,d\rho,
\end{align*} 
where $(\Lambda\Phi)^*$ is the adjoint operator of
$\Lambda\Phi:\mathcal U\to H^1(\mathbb T_L)$. 
By the Burkholder--Davis--Gundy inequality and Young's inequality, 
\begin{equation}
\begin{aligned}
\mathbb E
\sup_{0\le t\le T}|M_R(t)|
&\le
C C_\Phi^{1/2}
\mathbb E
\Big(
\int_0^{T\wedge\tau_R}
\|u(\rho)\|_{H^1}^2\,d\rho
\Big)^{1/2}\\
&\le
\frac12
\mathbb E
\sup_{0\le t\le T\wedge\tau_R}
\|u(t)\|_{H^1}^2
+
C_T,
\label{eq:pathwise_M_BDG}
\end{aligned}
\end{equation}
where $C_T$ is independent of $R$.
Taking the supremum in \eqref{eq:stopped_H1_energy} and applying both \eqref{eq:pathwise_M_BDG} and Young's inequality, we obtain $\sup\limits_{R\geq1} \mathbb E \sup\limits_{0\leq t\leq T\wedge\tau_R} \|u(t)\|_{H^1}^2 <\infty.$  
Since $\tau_R\uparrow\tau_\infty$, Fatou's lemma and
\eqref{eq:stopped_H1_energy} imply $\mathbb E
\left[
\sup\limits_{0\leq t<T\wedge\tau_\infty}
\|u(t)\|_{H^1}^2
\right]
<\infty$
for every fixed $T>0$. Consequently, $\sup\limits_{0\leq t<T\wedge\tau_\infty}
\|u(t)\|_{H^1}
<\infty, $ a.s. 
Taking a countable intersection over $T\in\mathbb N$, we obtain
$\sup\limits_{0\leq t<\tau_\infty}
\|u(t)\|_{H^1}
<\infty$ a.s. on $\{\tau_\infty<\infty\}.$
\iffalse
Since $\tau_R\uparrow\tau_\infty$, Fatou's lemma yields, for every
deterministic $T>0$, 
$\sup\limits_{0\le t<T\wedge\tau_\infty}
\|u(t)\|_{H^1}
<\infty$ a.s.
%\label{eq:pathwise_H1_local} 
Taking a countable intersection over $T\in\mathbb N$ gives 
$\sup\limits_{0\le t<\tau_\infty}
\|u(t)\|_{H^1}
<\infty$ a.s. on $\{\tau_\infty<\infty\}.$
\fi
%\label{eq:pathwise_H1_until_blowup}
The Sobolev embedding
$H^1(\mathbb T_L)\hookrightarrow L^\infty(\mathbb T_L)$
then yields \eqref{eq:pathwise_u_bound}.

% {\color{red}
% To obtain a pathwise differential inequality for the slope, we remove the additive stochastic forcing by introducing}
Now we introduce 
$\widetilde u:=u-Z$ and
$r:=\widetilde u_x=u_x-Z_x$.
Since
$d\widetilde u=F(u)\,dt$,
% differentiation with respect to $x$ gives
\begin{align}
r_t+ur_x
&=
-\frac12u_x^2
+
u^2
-
\Lambda
\left(
u^2+\frac12u_x^2
\right)
-uZ_{xx}.
\label{eq:pathwise_r_equation}
\end{align}
Fix $\omega$ such that $\tau_\infty(\omega)<\infty$, and choose an
integer $J>\tau_\infty(\omega)$. Set
\begin{align}
B
&:=
\sup_{0\le t<\tau_\infty}
\|u(t)\|_{L^\infty},
\quad 
K_J
:=
\sup_{0\le t\le J}
\left(
\|Z_x(t)\|_{L^\infty}
+
\|Z_{xx}(t)\|_{L^\infty}
\right).
\end{align}
By \eqref{eq:pathwise_u_bound} and
\eqref{eq:pathwise_Z_bound},
$B+K_J<\infty$. 
Along the characteristic flow, the positivity of the periodic Green
kernel yields 
$\frac{d}{dt}r(t,q(t,x))
\le
B^2+BK_J.$ 
According to $Z(0)=0$,
$r(0,x)=u_{0x}(x)$, we obtain 
$r(t,q(t,x))
\le
u_{0x}(x)
+
t(B^2+BK_J).$
%\label{eq:pathwise_upper_r_integrated} 
Since $q(t,\cdot)$ is a diffeomorphism,
\begin{align*}
\sup_{x\in\mathbb T_L}r(t,x)
&\le
\|u_{0x}\|_{L^\infty}
+
J(B^2+BK_J),
\qquad
0\le t<\tau_\infty.
%\label{eq:pathwise_r_upper_bound}
\end{align*}
Together with the boundedness of $Z_x$, this proves
\eqref{eq:pathwise_positive_slope}.

By the continuation criterion \eqref{eq:continuous_continuation},
on $\{\tau_\infty<\infty\}$ we have 
$\limsup\limits_{t\uparrow\tau_\infty}
\|u(t)\|_{W^{1,\infty}}
=\infty.$ 
Since \eqref{eq:pathwise_u_bound} and
\eqref{eq:pathwise_positive_slope} imply that
$\|u(t)\|_{L^\infty}$ and $\sup\limits_x u_x(t,x)$ remain bounded on
$[0,\tau_\infty)$, the blow-up can only occur through the negative
part of the slope. Hence
\begin{align}
\liminf_{t\uparrow\tau_\infty}
\inf_{x\in\mathbb T_L}u_x(t,x)
&=-\infty.
\label{eq:pathwise_slope_liminf}
\end{align}
It remains to show that the minimum slope actually tends to $-\infty$. %, %rather than merely along a subsequence. 
Set
$m(t):=\min\limits_{x\in\mathbb T_L}r(t,x)$.
According to 
$\partial_t\widetilde u=F(u)$ and
$r=\widetilde u_x$, we derive 
$\partial_t r
=
\partial_xF(u).$ 
Moreover, since $u\in H^s(\mathbb T_L)$ with $s>5/2$ and 
$F(u)
=
-uu_x-\Lambda\partial_x
\left(u^2+\frac12u_x^2\right),$ 
we obtain $F(u)\in H^{s-1}(\mathbb T_L)$ and 
$\partial_t r
=
\partial_xF(u)
\in H^{s-2}(\mathbb T_L)
$ 
on every compact subinterval of $[0,\tau_\infty)$.
As $\mathbb T_L$ is compact and $r(t,\cdot)\in C^1(\mathbb T_L)$,
the minimum of $r(t,\cdot)$ is attained for every $t$.
%Moreover, the local absolute continuity of $r$ as a $C(\mathbb T_L)$-valued function implies that $m(t):=\min\limits_{x\in\mathbb T_L}r(t,x)$ is locally absolutely continuous.
Let $m(t):=\min\limits_{x\in\mathbb T_L}r(t,x)$. 
% By the envelope theorem for minima over compact
% sets, 
Then for almost every $t$, there exists 
$x_t\in\operatorname*{argmin}\limits_{x\in\mathbb T_L}r(t,x)$ such that
$m'(t)=r_t(t,x_t)$.
Since $x_t$ is a spatial minimizer, % and $r(t,\cdot)\in C^1(\mathbb T_L)$,
we also get $r_x(t,x_t)=0$. 
Evaluating \eqref{eq:pathwise_r_equation} at $x_t$ gives
\begin{align*}
m'(t)
=&
-\frac12
\big(
m(t)+Z_x(t,x_t)
\big)^2
+
u(t,x_t)^2
-
\Lambda
\bigg(
u^2+\frac12u_x^2
\bigg)(t,x_t)
-
u(t,x_t)Z_{xx}(t,x_t).
\end{align*}
Denoting
\begin{align*}
K_{1,J}&:=\sup_{0\le t\le J}\|Z_x(t)\|_{L^\infty},\\
K_{2,J}&:=\sup_{0\le t\le J}\|Z_{xx}(t)\|_{L^\infty},\\
C_J&:=\frac12K_{1,J}^2+B^2+BK_{2,J},
\end{align*}
and using
$(a+b)^2\ge\frac12a^2-b^2$
%and the positivity of $\Lambda$ 
, we obtain 
$m'(t)
\le
-\frac14m(t)^2+C_J.$  
Since $r=u_x-Z_x$ and $Z_x$ remains bounded on
$[0,\tau_\infty)$, we have 
$m(t)
=
\min\limits_{x\in\mathbb T_L}r(t,x)
\le
\inf\limits_{x\in\mathbb T_L}u_x(t,x)
+
\|Z_x(t)\|_{L^\infty}.
$ 
Hence,  \eqref{eq:pathwise_slope_liminf} implies 
$\liminf\limits_{t\uparrow\tau_\infty}m(t)
=
-\infty.$ 
Therefore, we can choose $t_0<\tau_\infty$ such that 
$m(t_0)<-\sqrt{8C_J}.$ 
On the other hand, when 
$m(t)\le-\sqrt{8C_J}$, we have
$C_J\le \frac18m(t)^2$,  which implies 
\begin{align*}
m'(t)
&\le
-\frac14m(t)^2+C_J
\le
-\frac18m(t)^2
<0.
%\label{eq:pathwise_m_strict}
\end{align*}
Therefore, once $m$ enters the region
$(-\infty,-\sqrt{8C_J}]$, it remains in this region and is
strictly decreasing thereafter. 
\iffalse
Since $m$ is locally absolutely continuous, once $m$ falls below
$-\sqrt{8C_J}$, it cannot cross this level again from below.
Consequently,
$m(t)\le m(t_0)<-\sqrt{8C_J},$
for 
$t_0\le t<\tau_\infty,$ 
and $m$ is nonincreasing on $[t_0,\tau_\infty)$.
\fi
Combining this monotonicity with
$\liminf\limits_{t\uparrow\tau_\infty}m(t)=-\infty$,
we conclude that
$\lim\limits_{t\uparrow\tau_\infty}m(t)
=-\infty.$ 
Finally, since $u_x=r+Z_x$ and $Z_x$ remains bounded, for any
$x_t\in\operatorname*{argmin}\limits_{x\in\mathbb T_L}r(t,x)$,
\begin{align*}
\inf_{x\in\mathbb T_L}u_x(t,x)
&\le
u_x(t,x_t)
=
m(t)+Z_x(t,x_t)\le
m(t)+\|Z_x(t)\|_{L^\infty}.
\end{align*}
Therefore, 
$\lim\limits_{t\uparrow\tau_\infty}
\inf\limits_{x\in\mathbb T_L}u_x(t,x)
=
-\infty,$ 
which proves \eqref{eq:pathwise_negative_slope}.
\end{proof}
% {\color{red}

% The proposition shows that every finite time loss of regularity is
% necessarily of wave-breaking type. The solution amplitude remains
% bounded, as does the positive part of the slope, whereas the negative
% part of the slope diverges to $-\infty$.}

Now we provide a sufficient condition
for finite time wave breaking with positive probability in terms of
the initial slope.

\begin{theorem}
\label{thm:continuous_riccati}
Let $u_0\in H_0^s(\mathbb T_L)$, $s>5/2$, be deterministic, and assume
\begin{equation*}
\Phi\in\mathcal L_2(\mathcal U;H_0^{s-2}(\mathbb T_L)).
\end{equation*}
In addition, assume that for every deterministic $T'>0$ such that
$\mathbb P(\tau_\infty>T')=1$,
\begin{align}
\mathbb E
\sup_{0\le t\le T'}
\|u(t)\|_{H^s}^2
<\infty.
\label{eq:continuous_moment_assumption}
\end{align}
For $T>0$, set $C_2(T)
:=
2C_0^2
\left(
\|u_0\|_{H^1}^2+C_\Phi T
\right),$
where $C_0>0$ is the embedding constant in $\|u\|_{L^\infty}
\le C_0\|u\|_{H^1}.$
Suppose that there exists a deterministic $T>0$ and
$x_0\in\mathbb T_L$ such that
$y_0:=u_{0x}(x_0)<-\sqrt{C_2(T)}$
and the corresponding time $T_0$ defined by
$T_0
:=
\frac1{\sqrt{C_2(T)}}
\log
\left(
\frac{
\sqrt{C_2(T)}-y_0
}{
-y_0-\sqrt{C_2(T)}
}
\right)$.
If $T_0<T$, then
\begin{equation}
\mathbb P\left(
\{\tau_\infty\le T_0\}\cap\Omega_{\mathrm{wb}}
\right)>0,
\label{eq:positive_probability_wave_breaking}
\end{equation}
where $\Omega_{\mathrm{wb}}$ is the full-probability event
on which the pathwise wave breaking characterization in
Proposition~\ref{prop:pathwise_wave_breaking} holds. 
\end{theorem}
\iffalse
\begin{equation}
\begin{aligned}
\mathbb P\Big(&
\tau_\infty\le T_0
,\;
\sup_{0\le t<\tau_\infty}
\|u(t)\|_{L^\infty}<\infty,\\
&\sup_{0\le t<\tau_\infty}
\sup_{x\in\mathbb T_L} u_x(t,x)<\infty,\;
\lim_{t\uparrow\tau_\infty}
\inf_{x\in\mathbb T_L}u_x(t,x)
=-\infty
\Big)
>0.
\end{aligned}
\label{eq:positive_probability_wave_breaking}   
\end{equation}
\fi

\begin{proof}
Let $U(t):=u_x(t,q(t,x_0)),
y(t):=\mathbb E[U(t)].$
Assume, for contradiction, that $\mathbb P(\tau_\infty>T_0)=1.$
Then $U(t)$ is well defined on $[0,T_0]$ almost surely.
Applying \eqref{eq:continuous_moment_assumption} with $T'=T_0$ yields
\begin{align}
\sup_{0\le t\le T_0}|y(t)|
&\le
C
\mathbb E
\sup_{0\le t\le T_0}
\|u(t)\|_{H^s}
<\infty.
\label{eq:expected_slope_bounded}
\end{align}
On the other hand, 
by Jensen's inequality,
$\mathbb E[U(t)^2]
\ge
(\mathbb E[U(t)])^2 .$ 
Since $A(t)\ge0$, it follows that, for almost every
$t\in[0,T_0)$,
\begin{align}
y'(t)
&\le
-\frac12\mathbb E[U(t)^2]
+
\mathbb E
\left[
u(t,q(t,x_0))^2
\right]
\le
-\frac12y(t)^2
+
C_0^2
\mathbb E\|u(t)\|_{H^1}^2
\nonumber\\
&\le
-\frac12y(t)^2
+
C_0^2
\left(
\|u_0\|_{H^1}^2+C_\Phi T
\right)
=
-\frac12y(t)^2+\frac12C_2(T).
\label{eq:continuous_expected_riccati}
\end{align}

Let $c:=\sqrt{C_2(T)}$ and $z$ solve
\begin{align}
z'(t)
&=
-\frac12z(t)^2+\frac12c^2,
\qquad
z(0)=y_0.
\label{eq:continuous_comparison_ode}
\end{align} 
Since $y_0<-c$ and 
$z(t)
=
-c
\frac{
(c-y_0)+(-c-y_0)e^{ct}
}{
(c-y_0)-(-c-y_0)e^{ct}
},$ 
the denominator vanishes at the finite time $T_0
=
\frac1c
\log
\left(
\frac{c-y_0}{-y_0-c}
\right).$
Hence, $\lim\limits_{t\uparrow T_0}z(t)=-\infty.$ 
Set $w:=y-z$.
Due to \eqref{eq:continuous_expected_riccati} and
\eqref{eq:continuous_comparison_ode},
\begin{align*}
w'(t)
&\le
-\frac12
\bigl(
y(t)^2-z(t)^2
\bigr)
=
-\frac12
\bigl(
y(t)+z(t)
\bigr)w(t), \quad w(0)=0,
\end{align*}
for almost every $t<T_0$.
As a result, 
$\frac{d}{dt}
\left[
\exp
\left(
\frac12\int_0^t
(y(r)+z(r))\,dr
\right)
w(t)
\right]
\le0.$ 
It follows that
$w(t)\le0$, that is, 
$y(t)\le z(t),$
for 
$0\le t<T_0.$ 
Since $z(t)\to-\infty$ as $t\uparrow T_0$, this implies
$y(t)\to-\infty$, contradicting
\eqref{eq:expected_slope_bounded}.
Hence, the assumption
$\mathbb P(\tau_\infty>T_0)=1$ is impossible and  $\mathbb P(\tau_\infty\le T_0)>0.$

By Proposition~\ref{prop:pathwise_wave_breaking}, there exists a
full probability event $\Omega_{\mathrm{wb}}$ such that, on
$\Omega_{\mathrm{wb}}$, every finite time breakdown satisfies $\sup\limits_{0\le t<\tau_\infty}\|u(t)\|_{L^\infty}<\infty,$ $
\sup\limits_{0\le t<\tau_\infty}\sup_{x\in\mathbb T_L}u_x(t,x)<\infty,$
and $\lim\limits_{t\uparrow\tau_\infty}
\inf_{x\in\mathbb T_L}u_x(t,x)=-\infty.$ 
Since $\mathbb P(\Omega_{\mathrm{wb}})=1$ and
$\mathbb P(\tau_\infty\le T_0)>0$,
$\mathbb P\bigl(
\Omega_{\mathrm{wb}}\cap\{\tau_\infty\le T_0\}
\bigr)>0.$
This proves \eqref{eq:positive_probability_wave_breaking}.
\end{proof}

\iffalse
Proposition~\ref{prop:pathwise_wave_breaking} holds on an event of
probability one and identifies every finite-time breakdown as
bounded-amplitude wave breaking. Intersecting this event with
$\{\tau_\infty\le T_0\}$ yields
\eqref{eq:positive_probability_wave_breaking}.
\fi

\iffalse
Theorem~\ref{thm:continuous_riccati} shows that a sufficiently steep
negative initial slope leads to finite-time wave breaking with positive
probability. 
\fi

\section{Structure-preserving fully discrete method}
\label{sec:discretization}

In this section, we construct a structure-preserving fully discrete
method of \eqref{eq:SCH_original}. We first employ a
Fourier--Galerkin discretization in space and then apply a splitting
symplectic Runge--Kutta method in time.  The resulting method reproduces the discrete averaged energy law and
the multi-symplectic structure, while retaining the dominant
Riccati-type instability of the characteristic slope under the
regularity assumptions specified below.

For $N\in\mathbb N_+$, define %the zero-mean Fourier space
$S_N
:=
\big\{
v(x)=\sum\limits_{k=-N}^{N}\widehat v(k)
e^{\mathrm i k\mu(x-a)}:
\widehat v(0)=0,\quad \widehat v(-k)=\overline{\widehat v(k)}
\big\}
\subset H_0^s(\mathbb T_L),$
where $\mu=2\pi/L$. Let $P_N:L^2(\mathbb T_L)\to S_N$ be the
orthogonal projection 
$P_N f(x)
=
\sum\limits_{0<|k|\le N}\widehat f(k)
e^{\mathrm i k\mu(x-a)}.$ 
% For zero-mean $f$, this is the usual Fourier truncation to the modes
% $0<|k|\le N.$
\iffalse
{\color{red}
For $N\in\mathbb N_+$, let $S_N\subset H_0^s(\mathbb T_L)$ denote} the
space of real-valued trigonometric polynomials with Fourier modes
$|k|\le N$. Equivalently, 
each $u_N\in S_N$ can be represented as
\begin{align*}
u_N(t,x)
=
\sum\limits_{|k|\le N}
\widehat u_k(t)e^{\mathrm i k\mu(x-a)},
\qquad
\widehat u_{-k}(t)
=
\overline{\widehat u_k(t)}.
\end{align*}

{\color{red}Let $P_N:L^2(\mathbb T_L;\mathbb R)\to S_N$ 
denote the $L^2$-orthogonal projection onto $S_N$, defined by}
$P_Nf(x)
=
\sum_{|k|\le N}
\widehat f(k)e^{\mathrm i k\mu(x-a)}.$ 
\fi
We first consider a spatially semi-discrete method. 
Specifically, we seek $u_N(t)\in S_N$, with $u_N(0)=P_Nu_0$, such that for every $h\in S_N$,
\begin{align}
(du_N,h)
+
\big(
\frac12\partial_x(u_N^2),h
\big)dt
+
\big(
\partial_x\Lambda
\big(
u_N^2+\frac12u_{N,x}^2
\big),h
\big)dt
=
(P_N\Lambda\Phi\,dW(t),h), 
\label{eq:galerkin-weak}
\end{align}
% Since $P_N$ is the $L^2$-orthogonal projection onto $S_N$, we have $(f,h)=(P_Nf,h)$ for $f\in L^2(\mathbb T_L;\mathbb R)$ and $h\in S_N$.
% Therefore, \eqref{eq:galerkin-weak} 
which is equivalent to 
\begin{align}
du_N
+
P_N
\left[
\frac12\partial_x(u_N^2)
+
\partial_x\Lambda
\left(
u_N^2+\frac12u_{N,x}^2
\right)
\right]dt
=
P_N\Lambda\Phi\,dW(t).
\label{eq:galerkin-strong}
\end{align} 
Since $P_N$, $\Lambda$, and $\partial_x$ %are Fourier multiplier operators {\color{red}on the periodic domain,} they 
commute, %. Consequently,
we obtain 
\begin{align}
du_N
+
\left[
\frac12P_N\partial_x(u_N^2)
+
\partial_x\Lambda P_N
\left(
u_N^2+\frac12u_{N,x}^2
\right)
\right]dt
=
P_N\Lambda\Phi\,dW(t).
\label{eq:galerkin-strong-PN}
\end{align} 
Based on $F(u)
=
-u u_x
-\Lambda\partial_x
\left(
u^2+\frac12u_x^2
\right),$ 
the semi-discrete  system can be written as
\begin{align}
du_N
=
P_NF(u_N)\,dt
+
P_N\Lambda\Phi\,dW(t),
\qquad
u_N(0)=P_Nu_0,
\label{eq:galerkin-compact}
\end{align}
%For each fixed $N$, the projected equation \eqref{eq:galerkin-compact} is an $S_N$-valued finite-dimensional stochastic differential equation. 
whose drift coefficient is locally Lipschitz continuous on
$S_N$. 
It can be verified that \eqref{eq:galerkin-compact}  admits a unique maximal adapted solution 
$u_N$ up to a possible explosion time $\tau_N$.
In particular, when $\Phi=0$, \eqref{eq:galerkin-strong-PN} reduces to the semi-discrete deterministic multi-symplectic system 
\begin{align}
\label{Det;multi-symplecitc}
M\,d Z_N
+
K\partial_x Z_N\,dt
=
P_N\nabla_ZS_1(Z_N)dt,
\qquad
Z_N=(u_N,\varphi_N,w_N,v_N,p_N)^\top,
\end{align}
where $M^\top=-M$ and $K^\top=-K$. Its component form is
\begin{align*}
\left\{
\begin{aligned}
&\frac12d\varphi_N
-\frac12d p_N
-\partial_xv_Ndt
=
-w_Ndt
-
P_N
\left(
\frac32u_N^2+\frac12p_N^2
\right)dt,
\\
&-\frac12du_N+\partial_xw_Ndt=0,
\quad \partial_x\varphi_N=u_N,
\quad p_N=\partial_xu_N,\\
& \frac12d u_N
=
-P_N(u_Np_N)dt+v_Ndt.
\end{aligned}
\right.
\end{align*}
Eliminating the auxiliary variables  and employing the commutativity of $P_N$ with $\partial_x$ and $\Lambda$, we arrive at 
\begin{align*}
(1-\partial_x^2)d u_N
=
P_N
\left[
-3u_N\partial_xu_N
+
2(\partial_xu_N)(\partial_x^2u_N)
+
u_N\partial_x^3u_N
\right]dt.
\end{align*}

Let $N_t\in\mathbb N_+$ be the number of time steps and set
$
\tau=T/N_t,
t_k=k\tau,
k=0,\ldots,N_t.
$
%Here, $N$ denotes the Fourier truncation parameter, whereas $N_t$denotes the number of time steps. 
We denote by
$\mathbf Z_N^k
=
(U_N^k,\varphi_N^k,w_N^k,v_N^k,p_N^k)^\top$
the fully discrete approximation to $Z_N(t_k)$ with 
$U_N^0=P_Nu_0$,
$p_N^0=\partial_xU_N^0$, and
$\varphi_N^0=\partial_x^{-1}U_N^0$.
Thus
$\partial_x\varphi_N^0=U_N^0$ and
$\int_{\mathbb T_L}\varphi_N^0(x)\,dx=0$.
Furthermore,
$w_N^0=-P_N\bigl[\frac14(u_N^0)^2+
\frac12\Lambda((u_N^0)^2+\frac12(p_N^0)^2)\bigr]$ and
$v_N^0=P_N(u_N^0p_N^0)+\frac12F_N(u_N^0)$,
where
$F_N(u)=-P_N[u\,u_x+\partial_x\Lambda
(u^2+\frac12u_x^2)]$. 
% Using the algebraic constraints to eliminate $w$, $v$, $\varphi$,
% and $p$, we obtain the finite-dimensional ODE  $\frac{dU_N}{dt}=F_N(U_N)$
% for $U_N,$ 
% where $F_N(U)
% =
% -P_N\left[
% \frac{1}{2}\partial_x(U^2)
% +
% \partial_x\Lambda
% \left(
% U^2+\frac{1}{2}U_x^2
% \right)
% \right].$
On each interval $[t_k,t_{k+1}]$, 
we apply an implicit Runge--Kutta method to multi-symplectic system \eqref{Det;multi-symplecitc} and then combine with the exact solution of 
the additive stochastic subsystem. 
\iffalse
On each interval $[t_k,t_{k+1}]$, we first apply the Runge--Kutta
method to \eqref{Det;multi-symplecitc} and then integrate
the additive stochastic subsystem exactly. 
\fi
In detail, let
$\mathbf A=(a_{ij})_{i,j=1}^s,$ and 
$b=(b_1,\ldots,b_s)^\top,$ 
be the coefficients of an $s$-stage Runge--Kutta method. We assume that the Butcher matrix $\mathbf A$ is invertible, and
\begin{equation}\label{eq:RK-symplectic-condition}
b_ia_{ij}+b_ja_{ji}=b_ib_j,
\qquad \sum\limits_{i=1}^s b_i=1, \qquad
i,j=1,\ldots,s.
\end{equation} 
% {\color{red}The invertibility of $\mathbf A$ will be used below to show that the
% algebraic constraints are inherited at every Runge--Kutta stage.}
Based on the numerical approximation $\mathbf Z_N^k$ 
satisfying
\begin{align*}
\partial_x\varphi_N^k&=U_N^k,\qquad
p_N^k=\partial_xU_N^k,\\
\int_{\mathbb T_L}\varphi_N^k(x)\,dx&=0,
\end{align*}
we have
\begin{equation*}
\bar{\mathbf Z}_N^{k+1}
=\left(
\bar U_N^{k+1},
\bar\varphi_N^{k+1},
\bar w_N^{k+1},
\bar v_N^{k+1},
\bar p_N^{k+1}
\right)^\top
\end{equation*}
satisfying
\begin{align}
\left\{
\begin{aligned}
&\mathbf Z_{N,i}^k
=
\mathbf Z_N^k
+
\tau\sum\limits_{j=1}^s
a_{ij}\mathbf Q_{N,j}^k,\qquad i=1,\ldots,s,\\
&M\mathbf Q_{N,i}^k
+
K\partial_x\mathbf Z_{N,i}^k
=
P_N\nabla_ZS_1(\mathbf Z_{N,i}^k),
\qquad i=1,\ldots,s,\\
&\bar{\mathbf Z}_N^{k+1}
=
\mathbf Z_N^k
+
\tau\sum\limits_{i=1}^s
b_i\mathbf Q_{N,i}^k,
\end{aligned}
\right.
\label{eq:multisymplectic-RK-stages}
\end{align}
where $\mathbf Z_{N,i}^k$ and $\mathbf Q_{N,i}^k$ denote the
Runge--Kutta stage value and stage derivative, respectively.  
%By the standard local solvability theory for implicit Runge--Kutta methods, there exists, for sufficiently small $\tau$, a unique stage solution of \eqref{eq:multisymplectic-RK-stages} in a neighborhood of the current state. 
% The auxiliary stage variables are then uniquely recovered from the
% algebraic constraints and the zero-mean normalization. 
%The resulting stage variables are $\mathcal F_{t_k}$-measurable whenever $\mathbf Z_N^k$ is $\mathcal F_{t_k}$-measurable. 
Define 
$\mathbf Z_{N,i}^k
:=
(
U_{N,i}^k,
\varphi_{N,i}^k,
w_{N,i}^k,
v_{N,i}^k,
p_{N,i}^k
)^\top,$ and 
$\mathbf Q_{N,i}^k
:=
(
Q_{N,i}^{u,k},
Q_{N,i}^{\varphi,k},
Q_{N,i}^{w,k},$  $
Q_{N,i}^{v,k},
Q_{N,i}^{p,k}
)^\top.$ 
% The algebraic constraints are inherited at every Runge--Kutta stage.  
Then we have  
$\partial_x\varphi_{N,i}^k
=
U_{N,i}^k,$ 
$p_{N,i}^k
=
\partial_xU_{N,i}^k,$ for $i=1,\ldots,s.$ 
Since $U_{N,i}^k$ has zero spatial mean, the periodic relation
$\partial_x\varphi_{N,i}^k=U_{N,i}^k$ determines
$\varphi_{N,i}^k$ only up to an additive constant. We fix this
constant by imposing 
$\int_{\mathbb T_L}
\varphi_{N,i}^k(x)\,dx
=
0,$ $i=1,\ldots,s.$ 
Furthermore, the corresponding Runge--Kutta stage relations are
\begin{align*}
U_{N,i}^k
=
U_N^k
+
\tau\sum_{j=1}^s
a_{ij}Q_{N,j}^{u,k},
\;
\varphi_{N,i}^k
=
\varphi_N^k
+
\tau\sum_{j=1}^s
a_{ij}Q_{N,j}^{\varphi,k},\;
p_{N,i}^k
=
p_N^k
+
\tau\sum_{j=1}^s
a_{ij}Q_{N,j}^{p,k}.
\end{align*}
From the constraints at the Runge--Kutta stages it follows that 
\begin{align*}
\sum_{j=1}^s
a_{ij}
\left(
\partial_xQ_{N,j}^{\varphi,k}
-
Q_{N,j}^{u,k}
\right)
=
0,\quad 
\sum_{j=1}^s
a_{ij}
\left(
Q_{N,j}^{p,k}
-
\partial_xQ_{N,j}^{u,k}
\right)
=
0, \qquad i=1,\ldots,s.
\end{align*}
Since the Butcher matrix $\mathbf A=(a_{ij})_{i,j=1}^s$ is invertible,
we obtain 
$\partial_xQ_{N,i}^{\varphi,k}
=
Q_{N,i}^{u,k},$ 
$Q_{N,i}^{p,k}
=
\partial_xQ_{N,i}^{u,k}$, 
$i=1,\ldots,s.$ 
Eliminating the auxiliary stage variables gives
\begin{align*}
(1-\partial_x^2)Q_{N,i}^{u,k}
=
P_N
\left[
-3U_{N,i}^k\partial_xU_{N,i}^k
+
2(\partial_xU_{N,i}^k)
(\partial_x^2U_{N,i}^k)
+
U_{N,i}^k\partial_x^3U_{N,i}^k
\right].
\end{align*}
Equivalently,
\begin{align}
Q_{N,i}^{u,k}
=
-P_N
\left[
\frac12\partial_x\bigl((U_{N,i}^k)^2\bigr)
+
\partial_x\Lambda
\left(
(U_{N,i}^k)^2
+
\frac12(\partial_xU_{N,i}^k)^2
\right)
\right].
\label{eq:deterministic-stage-slope}
\end{align}
Since both $\varphi_N^k$ and $\varphi_{N,i}^k$ have zero spatial mean,
the Runge--Kutta stage relations and the invertibility of $\mathbf A$
imply 
$\int_{\mathbb T_L}
Q_{N,i}^{\varphi,k}(x)\,dx
=
0,$ $i=1,\ldots,s.$ 
Consequently, 
$\int_{\mathbb T_L}
\bar\varphi_N^{k+1}(x)\,dx
=
0.$ 
Moreover, we obtain
\begin{align}
&U_{N,i}^k
=
U_N^k
-
\tau\sum\limits_{j=1}^s
a_{ij}P_N
\left[
\frac12\partial_x
\bigl((U_{N,j}^k)^2\bigr)
+
\partial_x\Lambda
\left(
(U_{N,j}^k)^2
+
\frac12(\partial_xU_{N,j}^k)^2
\right)
\right],
\label{eq:deterministic-RK-stages}\\
&\bar U_N^{k+1}
=
U_N^k
-
\tau\sum\limits_{i=1}^s
b_iP_N
\left[
\frac12\partial_x
\bigl((U_{N,i}^k)^2\bigr)
+
\partial_x\Lambda
\left(
(U_{N,i}^k)^2
+
\frac12(\partial_xU_{N,i}^k)^2
\right)
\right],
\label{eq:deterministic-RK-update}
\end{align} 
% The stage-slope identities imply that the deterministic endpoint remains constraint-consistent: 
and $\partial_x\bar\varphi_N^{k+1}=\bar U_N^{k+1}$ and $\bar p_N^{k+1}=\partial_x\bar U_N^{k+1}$. 
For completeness, we justify the local solvability of the Runge--Kutta
stage equations. For fixed $N$, equation~(3.9) defines a finite-dimensional
nonlinear system on $S_N^s$. Let
$\mathbf{U}=(U_1,\ldots,U_s)\in S_N^s$ and define
$\mathcal{G}:S_N^s\times\mathbb{R}\to S_N^s$ componentwise by
\begin{align*}
\mathcal{G}_i(\mathbf{U},\tau)
&=
U_i-U_N^k
+
\tau\sum_{j=1}^s a_{ij}P_N
\left[
\frac{1}{2}\partial_x(U_j^2)
+
\partial_x\Lambda
\left(
U_j^2+\frac{1}{2}(\partial_xU_j)^2
\right)
\right],
\end{align*}
where $i=1,\ldots,s.$
At $\tau=0$, we have
\begin{align*}
\mathcal{G}
\bigl((U_N^k,\ldots,U_N^k),0\bigr)=0,
\quad 
D_{\mathbf{U}}\mathcal{G}
\bigl((U_N^k,\ldots,U_N^k),0\bigr)=I_{S_N^s}.
\end{align*}
Since $\mathcal{G}$ is smooth on the finite-dimensional space $S_N^s$,
the implicit function theorem yields, for sufficiently small $\tau$, a
unique stage solution $\{U_{N,i}^k\}_{i=1}^s$ in a neighborhood of
$(U_N^k,\ldots,U_N^k)$.
Moreover, whenever $Z_N^k$ is $\mathcal{F}{t_k}$-measurable, the locally unique stage variables are also $\mathcal{F}{t_k}$-measurable.

Let $G_N=P_N\Lambda\Phi$ and 
$\Delta W_k=W(t_{k+1})-W(t_k)$.
% Under the zero-mean assumption on the noise coefficient,
% $G_N\Delta W_k$ has zero spatial mean. Hence
% $\partial_x^{-1}G_N\Delta W_k$ is well defined.
% 
Based on the exact solution of the stochastic subsystem over $[t_k,t_{k+1}]$
% is additive and can therefore be integrated exactly through
\begin{align*}
dU_N
=
G_N\,dW(t),\;
d\varphi_N
=
\partial_x^{-1}G_N\,dW(t),\;
dw_N
=
0,\;
dv_N
=
0,\;
dp_N=
\partial_xG_N\,dW(t).
\end{align*} 
% Because all the coefficients in this subsystem are independent of the
% phase variables, 
we present the fully discrete method
\begin{align}
\left\{
\begin{aligned}
&U_N^{k+1}
=
\bar U_N^{k+1}
+
G_N\Delta W_k,\quad \varphi_N^{k+1}
=
\bar\varphi_N^{k+1}
+
\partial_x^{-1}G_N\Delta W_k,\\
&p_N^{k+1}
=
\bar p_N^{k+1}
+
\partial_xG_N\Delta W_k,\quad w_N^{k+1}
=
\bar w_N^{k+1},\quad
v_N^{k+1}
=
\bar v_N^{k+1}.
\end{aligned}
\right.
\label{eq:exact-stochastic-update}
\end{align}
As a result, 
% the stochastic update gives 
$\int_{\mathbb T_L}
\varphi_N^{k+1}(x)\,dx
=
0,$ and
%Moreover,
\begin{align*}
\partial_x\varphi_N^{k+1}
=
\partial_x\bar\varphi_N^{k+1}
+
G_N\Delta W_k
=
U_N^{k+1},\;
p_N^{k+1}
=
\bar p_N^{k+1}
+
\partial_xG_N\Delta W_k
=
\partial_xU_N^{k+1}.
\end{align*}
Hence the fully discrete method  preserves both algebraic constraints, i.e., 
$
\partial_x\varphi_N^{k+1}=U_N^{k+1}$ and $
p_N^{k+1}=\partial_xU_N^{k+1},
$
as well as the zero-mean normalization of $\varphi_N^{k+1}$.

\subsection{Discrete averaged energy law and stochastic geometric structure}
\label{subsec:discrete_structure} 
We first derive the discrete averaged energy evolution law of 
% the fully discrete method 
\eqref{eq:exact-stochastic-update}.
\begin{proposition}
\label{prop:fully_discrete_energy}
Assume that the  coefficients of the fully discrete method 
\eqref{eq:exact-stochastic-update} satisfy the symplectic condition 
\eqref{eq:RK-symplectic-condition}.  Then  \eqref{eq:exact-stochastic-update}  satisfies
\begin{align}
\mathbb E\|U_N^k\|_{H^1}^2
=
\mathbb E\|U_N^0\|_{H^1}^2
+t_k\|G_N\|_{\mathcal L_2(\mathcal U;H^1)}^2, \quad k=0,1,\ldots,N_t.
\label{eq:fully_discrete_energy}
\end{align}
\end{proposition}

\begin{proof}
For $F_N(U)=-P_N\left[\frac12\partial_x(U^2)+\partial_x\Lambda\left(U^2+\frac12(\partial_xU)^2\right)\right]$, 
using periodic integration by parts and the identity
$
\langle u,v\rangle_{H^1}
=((1-\partial_x^2)u,v),
$
one obtains $\langle U,F_N(U)\rangle_{H^1}=0.$
Based on \eqref{eq:RK-symplectic-condition}, we have $\|\bar U_N^{k+1}\|_{H^1}^2
=
\|U_N^k\|_{H^1}^2.$ 
By means of
$U_N^{k+1}=\bar U_N^{k+1}+G_N\Delta W_k$, we arrive at
\begin{align*}
\|U_N^{k+1}\|_{H^1}^2
={}&
\|U_N^k\|_{H^1}^2
+2\langle \bar U_N^{k+1},G_N\Delta W_k\rangle_{H^1}
+\|G_N\Delta W_k\|_{H^1}^2.
\end{align*}
Because %$\bar U_N^{k+1}$ is $\mathcal F_{t_k}$-measurable, 
the cross term
has zero expectation, 
$\mathbb E\|G_N\Delta W_k\|_{H^1}^2
=\tau\|G_N\|_{\mathcal L_2(\mathcal U;H^1)}^2.$ 
As a consequence,  $\mathbb E\|U_N^{k+1}\|_{H^1}^2
=
\mathbb E\|U_N^k\|_{H^1}^2
+
\tau
\|G_N\|_{\mathcal L_2(\mathcal U;H^1)}^2 .$
Summing this identity from $0$ to $k-1$ yields
\eqref{eq:fully_discrete_energy}. 
\end{proof}

%{\color{red}We next prove that the proposed fully discrete method preserves a discrete stochastic multi-symplectic  conservation law.}
Now we define
\begin{align*}
\Omega_N(\mathbf Z_N)
=
\frac12
\int_{\mathbb T_L}
\mathbf d\mathbf Z_N
\wedge
M\mathbf d\mathbf Z_N\,dx,\quad
\mathcal K_N(\mathbf Z_N)
=
\frac12
\mathbf d\mathbf Z_N
\wedge
K\mathbf d\mathbf Z_N,
\end{align*}
% {\color{red}where $\mathbf d$ denotes the exterior derivative with respect
% to the phase variables.}
and  set 
$\mathcal B_N
=(
G_N,
\partial_x^{-1}G_N,
0,0,
\partial_xG_N)^\top.$ 
Then the fully discrete method \eqref{eq:exact-stochastic-update} can be written as
\begin{align}
\mathbf Z_N^{k+1}
=
\bar{\mathbf Z}_N^{k+1}
+
\mathcal B_N\Delta W_k.
\label{eq:exact-additive-stochastic-flow}
\end{align}
%The operator $\mathcal B_N$ is independent of the phase variables.

\begin{theorem}
Suppose that the  coefficients of the fully discrete method 
\eqref{eq:exact-stochastic-update} satisfy the symplectic condition 
\eqref{eq:RK-symplectic-condition}. 
In addition, assume that the Butcher matrix
$\mathbf A$ is invertible and that $\tau$ is sufficiently small so
that the locally unique stage solution described above exists. Then, almost surely,
\begin{align*}
\Omega_N(\mathbf Z_N^{k+1})
-
\Omega_N(\mathbf Z_N^k)
+
\tau\sum_{i=1}^s
b_i
\int_{\mathbb T_L}
\partial_x
\mathcal K_N(\mathbf Z_{N,i}^k)\,dx
=
0,\quad k=0,1,\ldots,N_t-1.
\label{eq:integrated-discrete-multisymplectic-law}
\end{align*}
% Consequently,
% \begin{align}
% \Omega_N(\mathbf Z_N^{k+1})
% =
% \Omega_N(\mathbf Z_N^k),
% \qquad
% k=0,\ldots,N_t-1.
% \label{eq:global-spectral-symplectic-law}
% \end{align}
% Thus the method preserves a spatially integrated discrete
% multi-symplectic structure.
\end{theorem}
\begin{proof}
%We first consider the deterministic Runge--Kutta substep.
Taking the exterior derivative of the stage equations gives
\[
M\mathbf d\mathbf Q_{N,i}^k
+
K\partial_x\mathbf d\mathbf Z_{N,i}^k
=
P_N
\left[
\nabla_Z^2S_1(\mathbf Z_{N,i}^k)
\mathbf d\mathbf Z_{N,i}^k
\right].
\]
Since the Runge--Kutta coefficients satisfy the symplecticity
condition, %the standard symplectic Runge--Kutta argument yields
$$
\Omega_N(\bar{\mathbf Z}_N^{k+1})
-
\Omega_N(\mathbf Z_N^k)
=
\tau
\sum_{i=1}^s b_i
\int_{\mathbb T_L}
\mathbf d\mathbf Z_{N,i}^k
\wedge
M\mathbf d\mathbf Q_{N,i}^k\,dx .
$$
Because
$\nabla_Z^2S_1$ is symmetric and $P_N$ is an
orthogonal projection, we obtain
\begin{align*}
\mathbf d\mathbf Z_{N,i}^k
\wedge
M\mathbf d\mathbf Q_{N,i}^k
={}&
-
\mathbf d\mathbf Z_{N,i}^k
\wedge
K\partial_x\mathbf d\mathbf Z_{N,i}^k
+
\mathbf d\mathbf Z_{N,i}^k
\wedge
P_N
\left[
\nabla_Z^2S_1(\mathbf Z_{N,i}^k)
\mathbf d\mathbf Z_{N,i}^k
\right]\\
={}&
-
\mathbf d\mathbf Z_{N,i}^k
\wedge
K\partial_x\mathbf d\mathbf Z_{N,i}^k.
\end{align*}
Since $K^\top=-K$, $\mathbf d\mathbf Z_{N,i}^k
\wedge
K\partial_x\mathbf d\mathbf Z_{N,i}^k
=
\partial_x
\mathcal K_N(\mathbf Z_{N,i}^k).$
Therefore,
\begin{align}
\Omega_N(\bar{\mathbf Z}_N^{k+1})
-
\Omega_N(\mathbf Z_N^k)
+
\tau
\sum_{i=1}^s b_i
\int_{\mathbb T_L}
\partial_x
\mathcal K_N(\mathbf Z_{N,i}^k)\,dx
=
0 .
\label{eq:deterministic_discrete_multisymplectic}
\end{align}
Furthermore, for every fixed realization $\omega$, the increment
$\Delta W_k(\omega)$ is independent of the phase variables, which implies 
% $\mathbf d(\mathcal B_N\Delta W_k)=0,$
% and therefore 
$\mathbf d\mathbf Z_N^{k+1}
=
\mathbf d\bar{\mathbf Z}_N^{k+1}.$
Consequently, $\Omega_N(\mathbf Z_N^{k+1})
=
\Omega_N(\bar{\mathbf Z}_N^{k+1}),$
which together with
\eqref{eq:deterministic_discrete_multisymplectic}
proves the result.
\end{proof}

\subsection{Discrete Riccati mechanism}
\label{sec:fully_discrete_wave_breaking}
We next derive a discrete analogue of the characteristic-slope
Riccati mechanism.
In addition to \eqref{eq:RK-symplectic-condition}, we assume
that $b_i\ge0$ for $i=1,\ldots,s$.
Fix $x_0\in\mathbb T_L$ and set $X_N^0=x_0$. 
For each time step, consider the characteristic flow by applying
the same Runge--Kutta coefficients to the stage velocity fields
$U_{N,i}^k$. 
Thus, the internal characteristic stages and the
endpoint are 
\begin{align}
&X_{N,i}^k
=
X_N^k
+
\tau\sum_{j=1}^s
a_{ij}U_{N,j}^k(X_{N,j}^k),
\qquad i=1,\ldots,s,
\label{eq:discrete-characteristic-stages}\\
&X_N^{k+1}
=
X_N^k
+
\tau\sum_{i=1}^s
b_iU_{N,i}^k(X_{N,i}^k).
\label{eq:discrete-characteristic-update}
\end{align}

Assume that $u_0 \in H_0^3(\mathbb{T}_L)$ and $\Lambda \Phi \in L_2(U; H_0^3(\mathbb{T}_L)).$ 
For fixed $N$ and fixed stage functions $U_{N,i}^k$,
% the system \eqref{eq:discrete-characteristic-stages} is finite-dimensional.
since the Jacobian with respect to the stage variables is the identity
at $\tau=0$, the implicit function theorem yields a locally unique
solution of the system \eqref{eq:discrete-characteristic-stages} for sufficiently small $\tau$. 
As $U_{N,i}^k$ are $\mathcal F_{t_k}$-measurable, so are
$X_{N,i}^k$ and $X_N^{k+1}$. 
\iffalse
For fixed $N$ and fixed deterministic-stage functions
$U_{N,i}^k$, the system
\eqref{eq:discrete-characteristic-stages} is finite dimensional and
has a locally unique solution near $X_N^k$ for sufficiently small
$\tau$, by the implicit function theorem applied at $\tau=0$.
Since the deterministic stage values
are $\mathcal F_{t_k}$-measurable, the corresponding characteristic
stages and the endpoint in
\eqref{eq:discrete-characteristic-update} are also
$\mathcal F_{t_k}$-measurable. 
\fi
Define
$Y_N^k
:=
\partial_xU_N^k(X_N^k),$  
$V_{N,i}^k
:=
\partial_xU_{N,i}^k(X_{N,i}^k),$ 
and set 
$y_N^k=\mathbb E[Y_N^k]$ with $y_N^0=\partial_xP_Nu_0(x_0)$ and
$\eta_N^k=\sum\limits_{i=1}^s b_i\mathbb E[V_{N,i}^k]$.
For convenience, let
$h_{N,i}^k(x)
:=
(U_{N,i}^k(x))^2
+
\frac12(\partial_xU_{N,i}^k(x))^2.$  
From
\eqref{eq:deterministic-stage-slope}, the commutativity of $P_N$,
$\partial_x$, and $\Lambda$, and the identity
$\partial_x^2\Lambda=\Lambda-I$, we obtain 
$\partial_xQ_{N,i}^{u,k}
=
-\partial_xP_N
\left(
U_{N,i}^k\partial_xU_{N,i}^k
\right)
-
\partial_x^2\Lambda P_Nh_{N,i}^k.$ 
Adding
$U_{N,i}^k\partial_x^2U_{N,i}^k$ to both sides gives
\begin{align}
\partial_xQ_{N,i}^{u,k}
+
U_{N,i}^k\partial_x^2U_{N,i}^k
={}&
-\frac12(\partial_xU_{N,i}^k)^2
+
(U_{N,i}^k)^2
-
\Lambda h_{N,i}^k
+
\mathcal R_{N,i}^k,
\label{eq:fully-discrete-stage-slope-identity}
\end{align}
where the remainder  
$\mathcal R_{N,i}^k
=
\partial_x(I-P_N)
\left(
U_{N,i}^k\partial_xU_{N,i}^k
\right)
+
(\Lambda-I)(I-P_N)h_{N,i}^k.$ 
From  
$\|\partial_x(I-P_N)f\|_{L^\infty}
\le
CN^{-1/2}\|f\|_{H^2}$
and 
$\|(\Lambda-I)(I-P_N)g\|_{L^\infty}
\le
CN^{-1/2}\|g\|_{H^2},$ it follows that 
\begin{align}
\|\mathcal R_{N,i}^k\|_{L^\infty}
\le
CN^{-1/2}\|U_{N,i}^k\|_{H^3}^2.
\label{eq:fully-discrete-projection-remainder-estimate}
\end{align}
Making use of  
\begin{align*}
\partial_x\bar U_N^{k+1}(X_N^{k+1})
&=\partial_xU_N^k(X_N^{k+1})\\
&\quad+\tau\sum_{i=1}^s
b_i\partial_xQ_{N,i}^{u,k}(X_N^{k+1}),\\
\bar Y_N^{k+1}&=\partial_x\bar U_N^{k+1}(X_N^{k+1}),
\end{align*}
we deduce
\begin{align}
\bar Y_N^{k+1}-Y_N^k
={}&
\tau\sum_{i=1}^s b_i
\left[
\partial_xQ_{N,i}^{u,k}
+
U_{N,i}^k\partial_x^2U_{N,i}^k
\right](X_{N,i}^k)
+
\mathcal T_N^k,
\label{eq:fully-discrete-slope-balance}
\end{align}
with
% where the {\color{red}characteristic discretization remainder} is given exactly by
\begin{align}
\mathcal T_N^k
={}&
\partial_xU_N^k(X_N^{k+1})
-
\partial_xU_N^k(X_N^k)
-
\tau\sum_{i=1}^s b_i
U_{N,i}^k(X_{N,i}^k)
\partial_x^2U_{N,i}^k(X_{N,i}^k)
\nonumber\\
&+
\tau\sum_{i=1}^s b_i
\left[
\partial_xQ_{N,i}^{u,k}(X_N^{k+1})
-
\partial_xQ_{N,i}^{u,k}(X_{N,i}^k)
\right].
\label{eq:fully-discrete-temporal-remainder}
\end{align}
Furthermore, since
\begin{align}
Y_N^{k+1}
=
\bar Y_N^{k+1}
+
\bigl(\partial_xG_N\Delta W_k\bigr)(X_N^{k+1}),
\label{eq:fully-discrete-stochastic-slope-update}
\end{align}
and $X_N^{k+1}$ is $\mathcal F_{t_k}$-measurable, we derive
$\mathbb E
\left[
\bigl(\partial_xG_N\Delta W_k\bigr)(X_N^{k+1})
\bigm|
\mathcal F_{t_k}
\right]
=0.$
%The pointwise evaluation is well defined since $\partial_xG_N\in \mathcal L_2(\mathcal U;H^2(\mathbb T_L))$ and $H^2(\mathbb T_L)\hookrightarrow L^\infty(\mathbb T_L)$.

The following lemma estimates the consistency errors introduced by
the Runge--Kutta approximation of the characteristic slope.
% The condition $\tau N^{1/2}\le1$ is imposed to balance the temporal discretization error and the Fourier inverse inequality.}
\begin{lemma}
Suppose that $\tau N^{1/2}\le1$ and, for some $M_T>0$ independent of
$N$ and $\tau$,
\begin{align}
\sup_{0\le k\le N_t}
\mathbb E\|U_N^k\|_{H^3}^4
+
\sup_{0\le k<N_t}\sup_{1\le i\le s}
\mathbb E\|U_{N,i}^k\|_{H^3}^4
\le
M_T.
\label{eq:fully-discrete-uniform-H3-assumption}
\end{align}
Then there exists $C_T>0$, independent of $N$, $\tau$, and $k$, such
that
\begin{align}
&\mathbb E|V_{N,i}^k-Y_N^k|
\le
C_T\tau,
\label{eq:stage-slope-consistency}\\
&\mathbb E|\mathcal T_N^k|
\le
C_T\tau^2N^{1/2}.
\label{eq:temporal-remainder-estimate}
\end{align}
\end{lemma}

\begin{proof}
The stage relations and
\eqref{eq:discrete-characteristic-stages} imply
\begin{align*}
\|U_{N,i}^k-U_N^k\|_{H^2}
\leq&
\tau\sum_{j=1}^s
|a_{ij}|\,\|Q_{N,j}^{u,k}\|_{H^2}\le
C\tau\|U_{N,j}^k\|_{H^3}^2,\\
|X_{N,i}^k-X_N^k|
\leq&
\tau\sum_{j=1}^s
|a_{ij}|\,\|U_{N,j}^k\|_{L^\infty}.
\end{align*}
% {\color{red}
% By the boundedness of the Fourier projection in Sobolev spaces
% and the product estimate in one dimension, $\|Q_{N,j}^{u,k}\|_{H^2}
% \le
% C\|U_{N,j}^k\|_{H^3}^2.$Therefore,}
Using the decomposition
\[
\begin{aligned}
V_{N,i}^k-Y_N^k
=
(\partial_xU_{N,i}^k-\partial_xU_N^k)(X_{N,i}^k)
+
\partial_xU_N^k(X_{N,i}^k)
-\partial_xU_N^k(X_N^k),
\end{aligned}
\]
and by the Sobolev embedding $H^1(\mathbb T_L)\hookrightarrow
L^\infty(\mathbb T_L),$  we obtain
\begin{align*}
|V_{N,i}^k-Y_N^k|
\leq{}&
\|\partial_x(U_{N,i}^k-U_N^k)\|_{L^\infty}
+
\|\partial_x^2U_N^k\|_{L^\infty}
|X_{N,i}^k-X_N^k|.
\end{align*}
Taking expectations and using
\eqref{eq:fully-discrete-uniform-H3-assumption} proves
\eqref{eq:stage-slope-consistency}. 

Set
$\delta_N^k
:=
X_N^{k+1}-X_N^k.$ 
We obtain $\delta_N^k
=
\tau\sum\limits_{i=1}^s
b_iU_{N,i}^k(X_{N,i}^k)$, and 
\begin{align*}
\partial_xU_N^k(X_N^{k+1})
-
\partial_xU_N^k(X_N^k)
% =&
% \delta_N^k
% \int_0^1
% \partial_x^2U_N^k
% \left(
% X_N^k+\theta\delta_N^k
% \right)d\theta\\
=&
\tau\sum_{i=1}^s
b_iU_{N,i}^k(X_{N,i}^k)
\int_0^1
\partial_x^2U_N^k
\left(
X_N^k+\theta\delta_N^k
\right)d\theta.
\end{align*}
As a consequence,
\begin{align*}
&\partial_xU_N^k(X_N^{k+1})
-
\partial_xU_N^k(X_N^k)-
\tau\sum_{i=1}^s
b_iU_{N,i}^k(X_{N,i}^k)
\partial_x^2U_{N,i}^k(X_{N,i}^k)
\\
=&
\tau\sum_{i=1}^s
b_iU_{N,i}^k(X_{N,i}^k)
\left[
\int_0^1
\partial_x^2U_N^k
\left(
X_N^k+\theta\delta_N^k
\right)d\theta
-
\partial_x^2U_{N,i}^k(X_{N,i}^k)
\right].
\end{align*}
Making use of
$\partial_x^2U_N^k(X_{N,i}^k)$ gives
\begin{align*}
&\int_0^1
\partial_x^2U_N^k
\left(
X_N^k+\theta\delta_N^k
\right)d\theta
-
\partial_x^2U_{N,i}^k(X_{N,i}^k)
\\
=&
\int_0^1
\left[
\partial_x^2U_N^k
\left(
X_N^k+\theta\delta_N^k
\right)
-
\partial_x^2U_N^k(X_{N,i}^k)
\right]d\theta
+
\partial_x^2U_N^k(X_{N,i}^k)
-
\partial_x^2U_{N,i}^k(X_{N,i}^k),
\end{align*}
where
\begin{align*}
&\left|
\partial_x^2U_N^k(X_{N,i}^k)
-
\partial_x^2U_{N,i}^k(X_{N,i}^k)
\right|
\leq
\left\|
\partial_x^2
\left(
U_N^k-U_{N,i}^k
\right)
\right\|_{L^\infty},\\
&\left|
\partial_x^2U_N^k
\left(
X_N^k+\theta\delta_N^k
\right)
-
\partial_x^2U_N^k(X_{N,i}^k)
\right|
\leq
\|\partial_x^3U_N^k\|_{L^\infty}
\left|
X_N^k+\theta\delta_N^k-X_{N,i}^k
\right|.
\end{align*}
Since $0\leq\theta\leq1$, we derive 
$\left|
X_N^k+\theta\delta_N^k-X_{N,i}^k
\right|
\leq
|X_{N,i}^k-X_N^k|
+
|\delta_N^k|.$ 
Consequently,
\begin{align*}
&\left|
\int_0^1
\partial_x^2U_N^k
\left(
X_N^k+\theta\delta_N^k
\right)d\theta
-
\partial_x^2U_{N,i}^k(X_{N,i}^k)
\right|\\
\leq&
\left\|
\partial_x^2
\left(
U_N^k-U_{N,i}^k
\right)
\right\|_{L^\infty}
+
\|\partial_x^3U_N^k\|_{L^\infty}
\left(
|\delta_N^k|
+
|X_{N,i}^k-X_N^k|
\right).
\end{align*}
Furthermore, we arrive at $\|\partial_x^2(U_N^k-U_{N,i}^k)\|_{L^\infty}
\leq
C\tau N^{1/2}
\sum_{j=1}^s
|a_{ij}|\,\|Q_{N,j}^{u,k}\|_{H^2},$
\begin{align*}
&\|\partial_x^3U_N^k\|_{L^\infty}
\leq
CN^{1/2}\|U_N^k\|_{H^3},\quad 
\|\partial_x^2Q_{N,i}^{u,k}\|_{L^\infty}
\leq
CN^{1/2}\|Q_{N,i}^{u,k}\|_{H^2},\\
&\left|
\partial_xQ_{N,i}^{u,k}(X_N^{k+1})
-
\partial_xQ_{N,i}^{u,k}(X_{N,i}^k)
\right|\leq
\|\partial_x^2Q_{N,i}^{u,k}\|_{L^\infty}
|X_N^{k+1}-X_{N,i}^k|.
\end{align*}
Let
$A_N^k:=\|U_N^k\|_{H^3}$ and
$A_{N,i}^k:=\|U_{N,i}^k\|_{H^3}$. Then 
$|X_{N,i}^k-X_N^k|
\leq
C\tau\sum\limits_{j=1}^s|a_{ij}|A_{N,j}^k,$ 
$|\delta_N^k|
\leq
C\tau\sum\limits_{j=1}^s b_jA_{N,j}^k,$ and
$|X_N^{k+1}-X_{N,i}^k|
\leq
C\tau\sum\limits_{j=1}^s
(b_j+|a_{ij}|)A_{N,j}^k.$  
Substituting these bounds into
\eqref{eq:fully-discrete-temporal-remainder} yields
\begin{align*}
\mathbb E|\mathcal T_N^k|
\leq{}
C\tau^2N^{1/2}
\mathbb E
\Bigg[&
\sum_{i=1}^s b_iA_{N,i}^k
\Bigg(
\sum_{j=1}^s|a_{ij}|(A_{N,j}^k)^2
+
A_N^k
\sum_{j=1}^s
(b_j+|a_{ij}|)A_{N,j}^k
\Bigg)\\
&+
\sum_{i=1}^s b_i(A_{N,i}^k)^2
\sum_{j=1}^s
(b_j+|a_{ij}|)A_{N,j}^k
\Bigg].
\end{align*}
All terms on the right-hand side contain at most three factors of the
$H^3$ norms. For example, H\"older's inequality gives
\begin{align*}
\mathbb E[A_{N,i}^k(A_{N,j}^k)^2]
&\leq
\left(
\mathbb E(A_{N,i}^k)^2
\right)^{1/2}
\left(
\mathbb E(A_{N,j}^k)^4
\right)^{1/2},\\
\mathbb E[A_{N,i}^kA_N^kA_{N,j}^k]
&\leq
\left(
\mathbb E(A_{N,i}^k)^4
\right)^{1/4}
\left(
\mathbb E(A_N^k)^4
\right)^{1/4}
\left(
\mathbb E(A_{N,j}^k)^2
\right)^{1/2}.
\end{align*}
The remaining products are estimated in the same way. Therefore,
\eqref{eq:fully-discrete-uniform-H3-assumption} implies
$\mathbb E|\mathcal T_N^k|
\leq
C_T\tau^2N^{1/2},$
which proves \eqref{eq:temporal-remainder-estimate}.
\end{proof}

Let $C_0$ denote the norm of the Sobolev embedding
$H^1(\mathbb T_L)\hookrightarrow L^\infty(\mathbb T_L).$ Define the stage energy constant by 
$C_{2,\mathrm{st}}(T)
:=
2C_{0}^2
\sup\limits_{0\le k<N_t}
\sum\limits_{i=1}^s
b_i\mathbb E\|U_{N,i}^k\|_{H^1}^2.$
By the uniform regularity estimate
\eqref{eq:fully-discrete-uniform-H3-assumption}, the stage energy
constant satisfies
\begin{align*}
C_{2,\mathrm{st}}(T)
\le 2C_{0}^2
\sup_{0\le k<N_t}
\sum_{i=1}^s
b_i
\left(
\mathbb E\|U_{N,i}^k\|_{H^3}^4
\right)^{1/2}
\le
2C_{0}^2M_T^{1/2}.
\end{align*}
Hence $C_{2,\mathrm{st}}(T)$ is bounded independently of $N$ and
$\tau$.

\begin{theorem}
\label{thm:fully_discrete_riccati}
Let $\tau$ be sufficiently small so that the locally unique stages
constructed above exist. Assume further that
$b_i\ge0$ for $i=1,\ldots,s$,
$\tau N^{1/2}\le1$, and
$\Lambda\Phi\in
\mathcal L_2(\mathcal U;H_0^3(\mathbb T_L))$.
Suppose that the uniform regularity estimate
\eqref{eq:fully-discrete-uniform-H3-assumption} holds on $[0,T]$.
Then there exists $C_T>0$, depending on $T$, $M_T$, and the
Runge--Kutta coefficients, but independent of $N$ and $\tau$, 
such
that
\begin{align}
\frac{y_N^{k+1}-y_N^k}{\tau}
&\le
-\frac12(y_N^k)^2
+
\frac12C_{N,\tau}(T),
\qquad
k=0,\ldots,N_t-1,
\label{eq:fully-discrete-riccati-inequality}
\end{align}
where 
$C_{N,\tau}(T)
:=
C_{2,\mathrm{st}}(T)
+
C_T
\left(
2N^{-1/2}
+
\tau
+
2\tau N^{1/2}
\right).$ 
Let $y_{N,\tau}$ be the continuous, piecewise linear interpolation of
$\{y_N^k\}_{k=0}^{N_t}$, namely, 
$y_{N,\tau}(t)
=
y_N^k
+
\frac{t-t_k}{\tau}
\left(
y_N^{k+1}-y_N^k
\right)$ 
for 
$t\in[t_k,t_{k+1}].$ 
Define 
$\widehat C_{N,\tau}(T)
:=
C_{N,\tau}(T)
+
C_T\tau.$
Then, for almost every $t\in(0,T)$,
\begin{align}
y_{N,\tau}'(t)
&\le
-\frac12y_{N,\tau}(t)^2
+
\frac12\widehat C_{N,\tau}(T).
\label{eq:interpolated-fully-discrete-riccati-inequality}
\end{align} 
If
$y_N^0<-\sqrt{\widehat C_{N,\tau}(T)}$, define 
$T_{N,\tau}^*
=
\frac{1}{\sqrt{\widehat C_{N,\tau}(T)}}
\log
\left(
\frac{
\sqrt{\widehat C_{N,\tau}(T)}-y_N^0
}{
-y_N^0-\sqrt{\widehat C_{N,\tau}(T)}
}
\right).$ 
Then the uniform regularity estimate
\eqref{eq:fully-discrete-uniform-H3-assumption} can hold on $[0,T]$
only if $T_{N,\tau}^*>T$.  %Equivalently, if
% $T_{N,\tau}^*\le T$, the $N$- and $\tau$-uniform bound
% \eqref{eq:fully-discrete-uniform-H3-assumption} must fail before time
% $T$.
\end{theorem}

\begin{proof}
Substituting
\eqref{eq:fully-discrete-stage-slope-identity} into
\eqref{eq:fully-discrete-slope-balance}, using
\eqref{eq:fully-discrete-stochastic-slope-update}, and taking
expectations, we derive 
\begin{align*}
\frac{y_N^{k+1}-y_N^k}{\tau}
={}&
-\frac12
\sum_{i=1}^s
b_i\mathbb E[(V_{N,i}^k)^2]
+
\sum_{i=1}^s
b_i\mathbb E
\left[
(U_{N,i}^k(X_{N,i}^k))^2
\right]
\\
&-
\sum_{i=1}^s
b_i\mathbb E
\left[
(\Lambda h_{N,i}^k)(X_{N,i}^k)
\right]
+
\sum_{i=1}^s
b_i\mathbb E
\left[
\mathcal R_{N,i}^k(X_{N,i}^k)
\right]
+
\frac1\tau
\mathbb E[\mathcal T_N^k].
\end{align*}
Since $h_{N,i}^k\ge0$ and the periodic Green function of $\Lambda$ is
positive,
$\Lambda h_{N,i}^k\ge0$.
Moreover, the Sobolev embedding
$H^1(\mathbb T_L)\hookrightarrow L^\infty(\mathbb T_L)$ gives
\begin{align*}
\sum_{i=1}^s
b_i\mathbb E
\left[
(U_{N,i}^k(X_{N,i}^k))^2
\right]
&\le
C_{0}^2
\sum_{i=1}^s
b_i\mathbb E\|U_{N,i}^k\|_{H^1}^2\le
\frac12C_{2,\mathrm{st}}(T).
\end{align*}
By
\eqref{eq:fully-discrete-projection-remainder-estimate} and
\eqref{eq:temporal-remainder-estimate}, we obtain 
$\left|
\sum\limits_{i=1}^s
b_i\mathbb E
\left[
\mathcal R_{N,i}^k(X_{N,i}^k)
\right]
\right|
\le
C_TN^{-1/2},$
and 
$\frac1\tau
\left|
\mathbb E[\mathcal T_N^k]
\right|
\le
C_T\tau N^{1/2}.$ 
%Because $b_i\ge0$ and $\sum\limits_{i=1}^s b_i=1$, 
Based on 
Jensen's inequality, 
$\sum\limits_{i=1}^s
b_i\mathbb E[(V_{N,i}^k)^2]
\ge
\big(
\sum\limits_{i=1}^s
b_i\mathbb E[V_{N,i}^k]
\big)^2
=
(\eta_N^k)^2,$ 
and thus
\begin{align*}
\frac{y_N^{k+1}-y_N^k}{\tau}
\le
-\frac12(\eta_N^k)^2
+
\frac12C_{2,\mathrm{st}}(T)
+
C_T
\left(
N^{-1/2}
+
\tau N^{1/2}
\right).
\end{align*} 
By \eqref{eq:stage-slope-consistency},
we arrive at
$|\eta_N^k-y_N^k|
\le
\sum\limits_{i=1}^s
b_i
\mathbb E|V_{N,i}^k-Y_N^k|
\le
C_T\tau.$ 
The uniform regularity assumption and the Sobolev embedding also give 
$|\eta_N^k|+|y_N^k|
\le
C_T.$ 
Consequently,
\begin{align*}
-(\eta_N^k)^2
&\le
-(y_N^k)^2
+
|\eta_N^k-y_N^k|
\left(
|\eta_N^k|+|y_N^k|
\right)
\le
-(y_N^k)^2
+
C_T\tau.
\end{align*}
Combining the preceding estimates 
proves
\eqref{eq:fully-discrete-riccati-inequality}.

The slope balance, the uniform $H^3$ bound, and the
remainder estimates give 
$\frac{|y_N^{k+1}-y_N^k|}{\tau}
\le
C_T\bigl(1+N^{-1/2}+\tau N^{1/2}\bigr).$ 
Under $\tau N^{1/2}\le1$, it follows that
$|y_N^{k+1}-y_N^k|\le C_T\tau$. Hence, for
$t\in[t_k,t_{k+1}]$,
$|y_{N,\tau}(t)-y_N^k|\le C_T\tau$ and
$|y_{N,\tau}(t)^2-(y_N^k)^2|\le C_T\tau$.
Since $y_{N,\tau}'=(y_N^{k+1}-y_N^k)/\tau$ almost everywhere on $(t_k,t_{k+1})$, \eqref{eq:fully-discrete-riccati-inequality} yields 
$y_{N,\tau}'(t)
\le
-\frac12y_{N,\tau}(t)^2
+\frac12\widehat C_{N,\tau}(T),$ 
which proves \eqref{eq:interpolated-fully-discrete-riccati-inequality}. 
Let $z_{N,\tau}$ solve 
$z_{N,\tau}'
=
-\frac12z_{N,\tau}^2
+
\frac12\widehat C_{N,\tau}(T),$
with 
$z_{N,\tau}(0)=y_N^0.$ 
Since the vector field
$f(z)=-\frac12z^2+\frac12\widehat C_{N,\tau}(T)$
is locally Lipschitz, the scalar comparison principle implies 
$y_{N,\tau}(t)\le z_{N,\tau}(t)$ for all times at which
$z_{N,\tau}$ is finite. Set
$c_{N,\tau}=\sqrt{\widehat C_{N,\tau}(T)}$.
If $y_N^0<-c_{N,\tau}$, separation of variables gives 
$z_{N,\tau}(t)
=
-c_{N,\tau}
\frac{
(c_{N,\tau}-y_N^0)
+
(-c_{N,\tau}-y_N^0)e^{c_{N,\tau}t}
}{
(c_{N,\tau}-y_N^0)
-
(-c_{N,\tau}-y_N^0)e^{c_{N,\tau}t}
}.$ 
Its denominator vanishes at $T_{N,\tau}^*$, and $z_{N,\tau}(t)\to-\infty$ as $t\uparrow T_{N,\tau}^*$.
On the other hand,
\eqref{eq:fully-discrete-uniform-H3-assumption} implies 
$|y_{N,\tau}(t)|
\le
C
\sup\limits_{0\le k\le N_t}
\left(
\mathbb E\|U_N^k\|_{H^3}^2
\right)^{1/2}
\le
C_T.$ 
If $T_{N,\tau}^*\le T$, then
$z_{N,\tau}(t)\to-\infty$ as
$t\uparrow T_{N,\tau}^*$, whereas the uniform regularity assumption
keeps $y_{N,\tau}$ bounded. This contradicts
$y_{N,\tau}(t)\le z_{N,\tau}(t)$ for
$t<T_{N,\tau}^*$. Hence $T_{N,\tau}^*>T$. 
\end{proof}

Theorem \ref{thm:fully_discrete_riccati} shows that the fully discrete characteristic slope retains
the dominant Riccati-type instability of its continuous counterpart.
If the comparison time $T^*_{N,\tau}$ lies in $[0,T]$, the assumed
uniform $H^3$-regularity estimate cannot persist on the whole interval.
Thus, sufficiently negative discrete slopes create an obstruction to
uniform high-regularity control through the same negative quadratic
mechanism that drives continuous wave breaking.
Moreover, the additional terms in
the discrete slope inequality arise from spatial projection and
characteristic discretization and become negligible under mesh refinement
whenever 
$N^{-1/2}\to0$, 
$\tau\to0$, and
$\tau N^{1/2}\to0.$ 
This quantifies the consistency of the discrete Riccati mechanism with
its continuous counterpart.

\section{Numerical experiments}
\label{sec:numerics}
This section complements the theoretical analysis through numerical
experiments. 
We use two smooth zero-mean initial conditions for complementary
purposes: a trigonometric profile for temporal convergence and
averaged energy tests on a smooth pre-breaking interval, and a localized
pulse with a pronounced descending flank to visualize pathwise
slope steepening. We consider \(\mathbb T_{2\pi}=\mathbb R/(2\pi\mathbb Z)\), identified
with the computational interval \([0,2\pi)\).  Throughout, $N$ denotes the Fourier cutoff defining $S_N$, with $N_{\mathrm F}=2N+1$ Fourier modes and equally spaced grid points. The zero Fourier coefficient is kept at zero, consistent with the zero-mean setting.  For the first two tests,
we take $u_0(x)=0.4\sin x+0.2\sin(2x)-0.1\sin(3x).$
The noise operator is truncated to the first \(50\) nonzero Fourier modes
in each of the cosine and sine components and is defined by
$\Phi e_m^{\mathrm c}
=\varepsilon(1+\mu_m^2)^{-1}\sqrt{\frac{2}{L}}\cos(\mu_m x),
\Phi e_m^{\mathrm s}
=\varepsilon(1+\mu_m^2)^{-1}\sqrt{\frac{2}{L}}\sin(\mu_m x),$
where \(\mu_m=2\pi m/L\), \(m=1,\ldots,50\), and \(L=2\pi\). 
\iffalse
In this section, we verify the theoretical results via the numerical experiments. 
In detail, we consider the periodic domain $\mathbb T_{2\pi}=\mathbb R/(2\pi\mathbb Z),$
with $[0,2\pi)$ as a representative computational interval, and use
the smooth, zero-mean initial value
$u_0(x)=0.4\sin x+0.2\sin(2x)-0.1\sin(3x).$ 
In addition, we take
$\Phi e_m^{\mathrm c}
=\varepsilon(1+\mu_m^2)^{-1}\sqrt{2/L}\cos(\mu_mx)$ and
$\Phi e_m^{\mathrm s}
=\varepsilon(1+\mu_m^2)^{-1}\sqrt{2/L}\sin(\mu_mx)$ with $\mu_m=2\pi m/L$ and 
$m=1,\ldots,50$. 
\fi

% {\color{red}
% All convergence tests are performed on a time interval
% $[0,T]$ satisfying
% $T<T_b^{\mathrm{num}},$
% where $T_b^{\mathrm{num}}$ denotes the numerically observed
% wave breaking time. Thus, the reported errors characterize the
% standard discretization convergence in the regular regime and are
% not affected by the singular behavior associated with wave breaking.}
% \iffalse
% The convergence experiments are carried out on time intervals that remain
% strictly before the numerically observed loss of regularity, so that they
% measure ordinary refinement behavior rather than singular-time effects. 
% \fi
\begin{figure}[h]\centering\includegraphics[width=\linewidth]{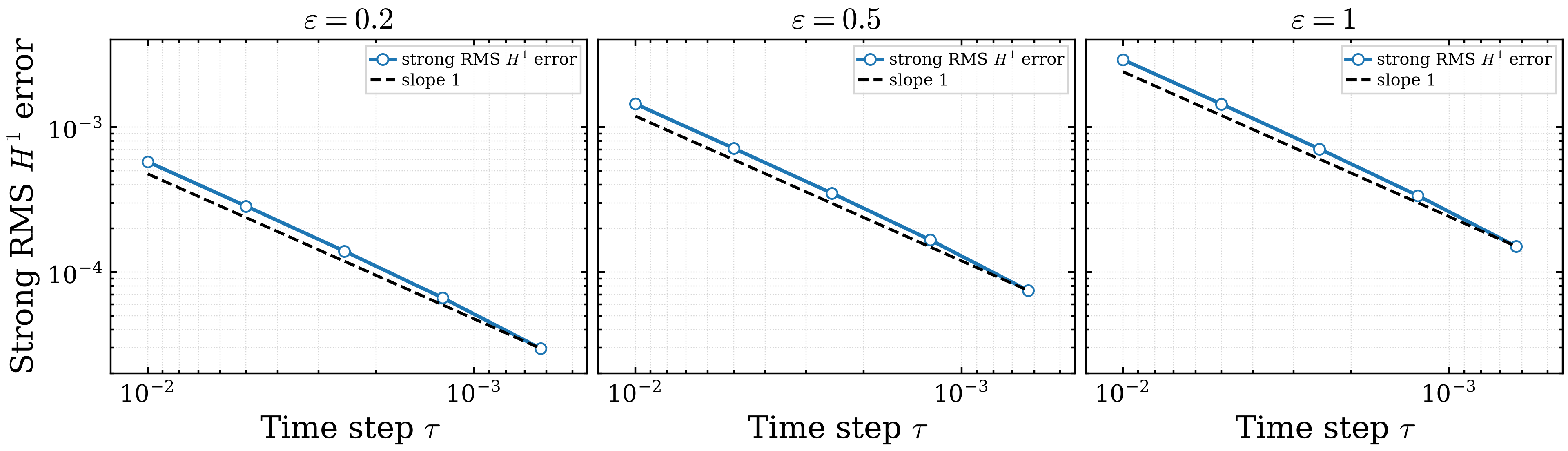}\caption{Strong temporal convergence error for \(\varepsilon=0.2,0.5,1\).}\label{fig:num-convergence}\end{figure}

Figure~\ref{fig:num-convergence} reports the reference-based strong temporal self-convergence for \(\varepsilon=0.2,0.5,1\) with \(N_{\mathrm F}=341\), corresponding to the Fourier cutoff \(N=170\). We take \(T=0.5\), \(\tau=10^{-2}2^{-j}\) for \(j=0,\ldots,4\), the reference time step \(\tau_{\rm ref}=1.5625\times10^{-4}\), and \(M=2000\) independent sample paths. With $U_{N,\tau_{\rm ref}}^{(\ell)}$ denoting the reference solution, the strong root-mean-square error is \(\mathcal E_{\tau}(T)=[M^{-1}\sum\limits_{\ell   =1}^M\|U_{N,\tau}^{(\ell)}(T)-U_{N,\rm ref}^{(\ell)}(T)\|_{H^1}^2]^{1/2}\). For each sample path, the coarse increments are sums of the corresponding reference increments. 
The least-squares slopes fitted to the five log--log data points are \(1.0658\), \(1.0656\), and \(1.0656\), respectively, so the observed temporal order is approximately one.

\iffalse
The computation is performed on
\(\mathbb T_{2\pi}=\mathbb R/(2\pi\mathbb Z)\), represented by
\([0,2\pi)\). 
\fi

% \begin{table}[t]
% \centering
% \caption{Numerical verification of the expected discrete energy growth law
% using \(2000\) sample paths for each noise intensity.}
% \label{tab:num-energy}
% {\scriptsize
% \renewcommand{\arraystretch}{1.08}
% \begin{tabular*}{\linewidth}
% {@{\extracolsep{\fill}}rrrrrr@{}}
% \toprule
% \(\varepsilon\)
% & \(k_{\mathrm{th}}\)
% & \(k_{\mathrm{fit}}\)
% & Rel. diff. (\%)
% & \(\overline{\Delta\mathcal E}_N(T)\)
% & \(k_{\mathrm{th}}T\) \\
% \midrule
% \(0.1\)
% & \(1.342961\times10^{-3}\)
% & \(1.359698\times10^{-3}\)
% & \(1.2462\)
% & \(6.800180\times10^{-2}\)
% & \(6.714806\times10^{-2}\) \\
% \(0.2\)
% & \(5.371845\times10^{-3}\)
% & \(5.361453\times10^{-3}\)
% & \(0.1935\)
% & \(2.675841\times10^{-1}\)
% & \(2.685922\times10^{-1}\) \\
% \(0.5\)
% & \(3.357403\times10^{-2}\)
% & \(3.326800\times10^{-2}\)
% & \(0.9115\)
% & \(1.667515\)
% & \(1.678702\) \\
% \(1\)
% & \(1.342961\times10^{-1}\)
% & \(1.338436\times10^{-1}\)
% & \(0.3370\)
% & \(6.694614\)
% & \(6.714806\) \\
% \bottomrule
% \end{tabular*}
% }
% \end{table}

\begin{figure}[t]
\centering
\includegraphics[width=1\linewidth]
{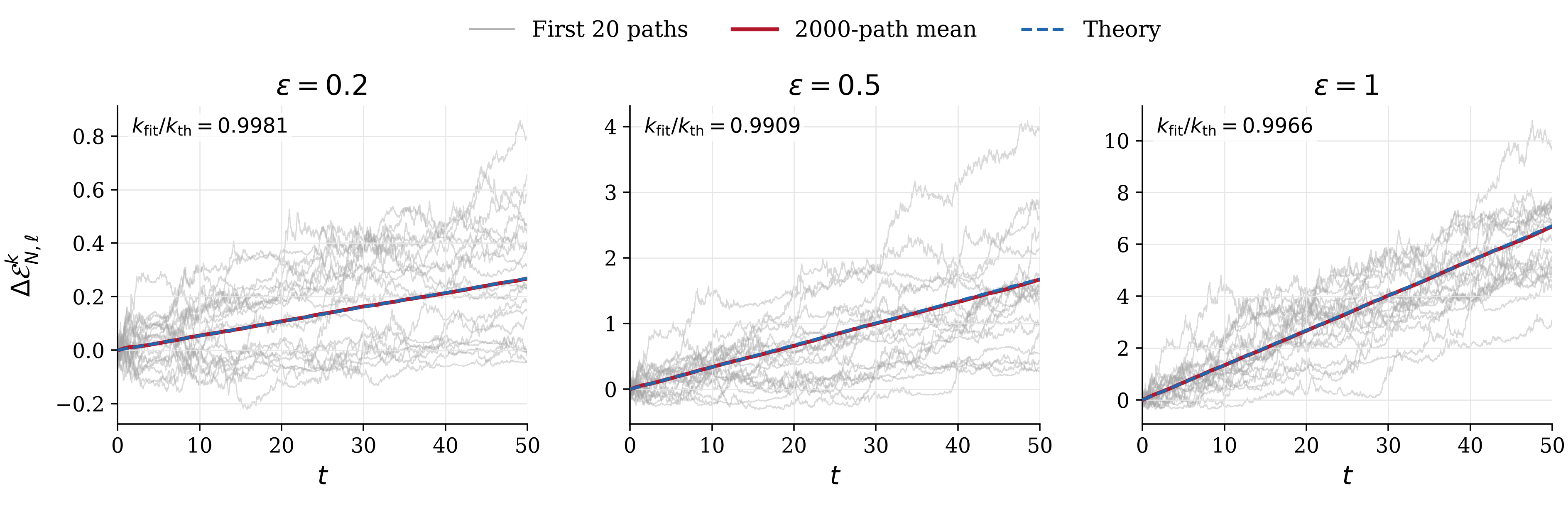}
\caption{The evolution law of discrete averaged energy increment for the fully discrete method.
% \(\varepsilon=0.1,0.2,0.5,1\). 
% The gray curves show the first \(20\)
% individual trajectories, the red curves are the sample means over all
% \(2000\) paths, and the blue dashed curves are the theoretical growth laws
% from \eqref{eq:num-energy-theoretical-slope}.
}
\label{fig:num-energy}
\end{figure}

% Table~\ref{tab:num-energy} and 
Figure~\ref{fig:num-energy} examines the discrete averaged energy
evolution.  For each noise intensity \(\varepsilon=0.2\), \(0.5\), and
\(1\), we compute the error 
$\Delta\mathcal E_N^k
=
\frac{1}{M}\sum\limits_{\ell=1}^{M}
\Big(
\mathcal E_N(U_{N,\ell}^k)-$
$
\mathcal E_N(U_{N,\ell}^0)
\Big).$
where \(N=170\), \(\tau=0.0025\), \(T=50\), and
\(M=2000\).  We fit
\(\Delta\mathcal E_N^k\approx k_{\mathrm{fit}}t_k\) by least squares
subject to zero intercept, where $k_{\mathrm{fit}}
=\sum\limits_{k=0}^{N_t}t_k\Delta\mathcal E_N^k/
\sum\limits_{k=0}^{N_t}t_k^2.$
The theoretical slope is
$k_{\mathrm{th}}
=
\varepsilon^2\sum\limits_{m=1}^{50}(1+\mu_m^2)^{-3}.$
Here, the power \((1+\mu_m^2)^{-3}\) results from the Helmholtz
operator in the velocity formulation together with the \(H^1\)-energy
weighting.  The ratios \(k_{\mathrm{fit}}/k_{\mathrm{th}}\) are
\(0.9981\), \(0.9909\), and \(0.9966\), respectively.  These results are
consistent with the predicted linear-in-time growth of the expected
discrete energy and its quadratic dependence on the noise intensity as shown in
Proposition~\ref{prop:fully_discrete_energy}.

\iffalse
Figure~\ref{fig:num-energy} shows the averaged
energy increment
${\Delta\mathcal E}_N^k
=M^{-1}\sum\limits_{\ell=1}^M
\bigl(\mathcal E_N(U_{N,\ell}^k)-\mathcal E_N(U_N^0)\bigr)$ 
as a function of time for the three noise intensities
$\varepsilon=0.2,0.5,1$.
For each noise intensity, we fit
${\Delta\mathcal E}_N^k\approx k_{\mathrm{fit}}t_k$
by a least-squares line constrained to pass through the origin, where
$k_{\mathrm{fit}}
=(\sum_k t_k{\Delta\mathcal E}_N^k)/(\sum_k t_k^2)$. 
The fitted slopes agree closely with the theoretical prediction
$k_{\mathrm{th}}
=\varepsilon^2\sum_{m=1}^{50}(1+\mu_m^2)^{-3},$ since 
the ratios $k_{\mathrm{fit}}/k_{\mathrm{th}}$ are
$0.9981$, $0.9909$, and $0.9966$, respectively.
%{\color{blue}The corresponding relative error is measured by $\mathrm{Err}_{\mathrm{rel}}=|k_{\mathrm{fit}}-k_{\mathrm{th}}|/k_{\mathrm{th}}\times100\%$.}
The close agreement between the fitted and theoretical slopes shows
that the numerical method reproduces both the linear growth of the
expected energy in time and the quadratic dependence of its growth rate
on the noise intensity. The computations use the initial data above
with $N=341$, $\tau=0.0025$, and $T=50$, and the Monte Carlo averages
are computed from $M=2000$ independent sample paths for each value of
$\varepsilon$.
\fi

\begin{figure}[!htbp] \centering \includegraphics[width=0.45\linewidth]{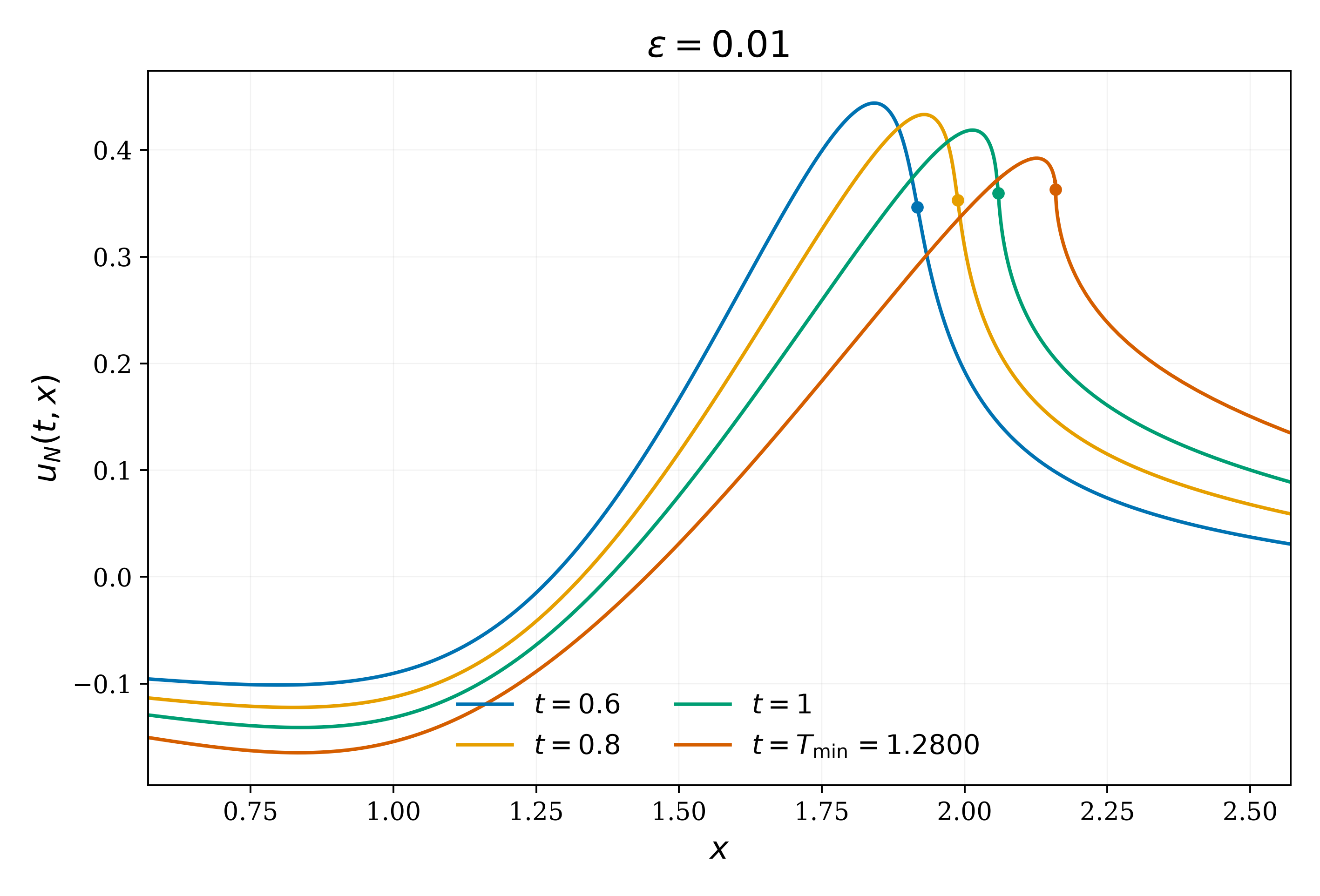}  \includegraphics[width=0.45\linewidth]{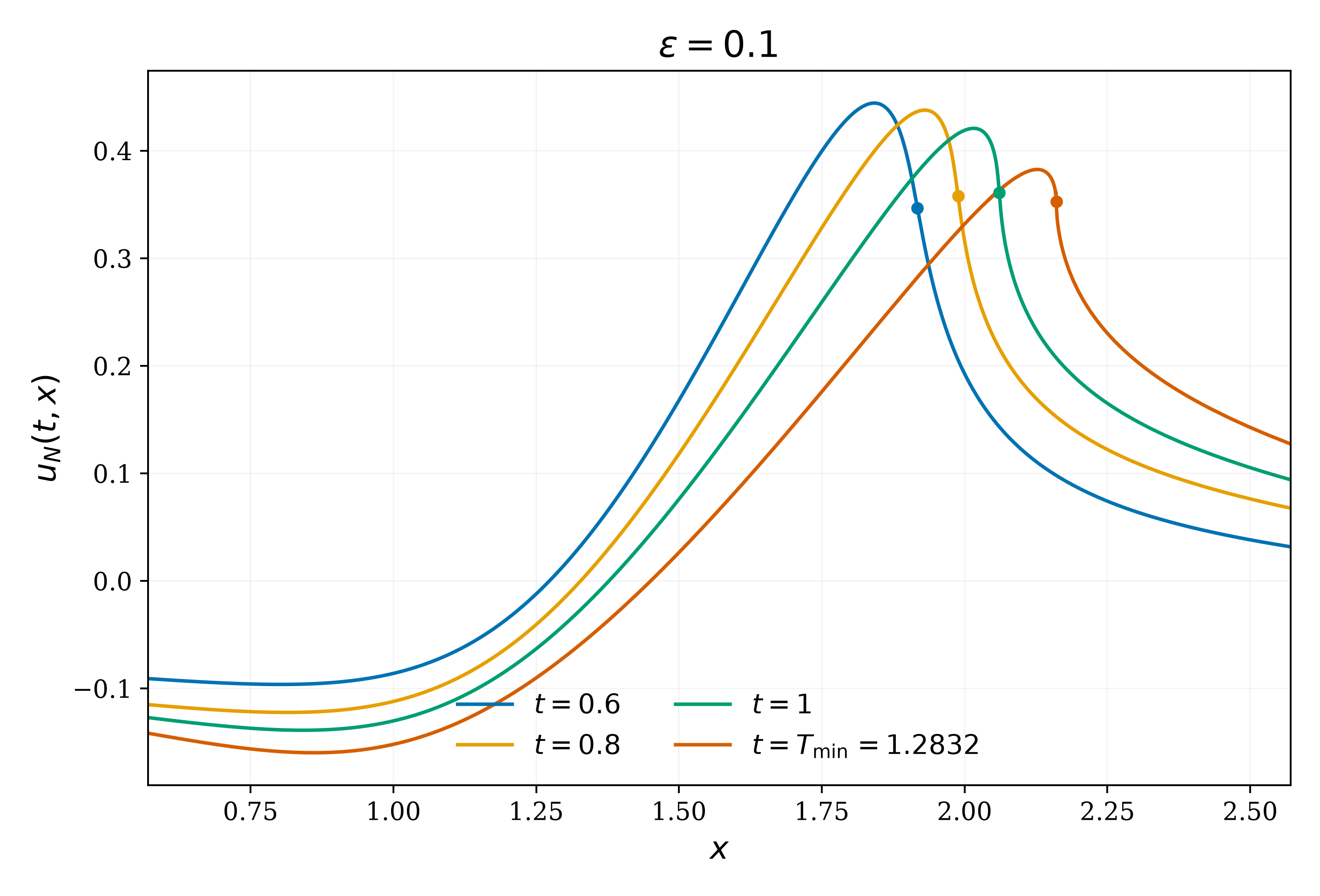}\\
\includegraphics[width=0.45\linewidth]{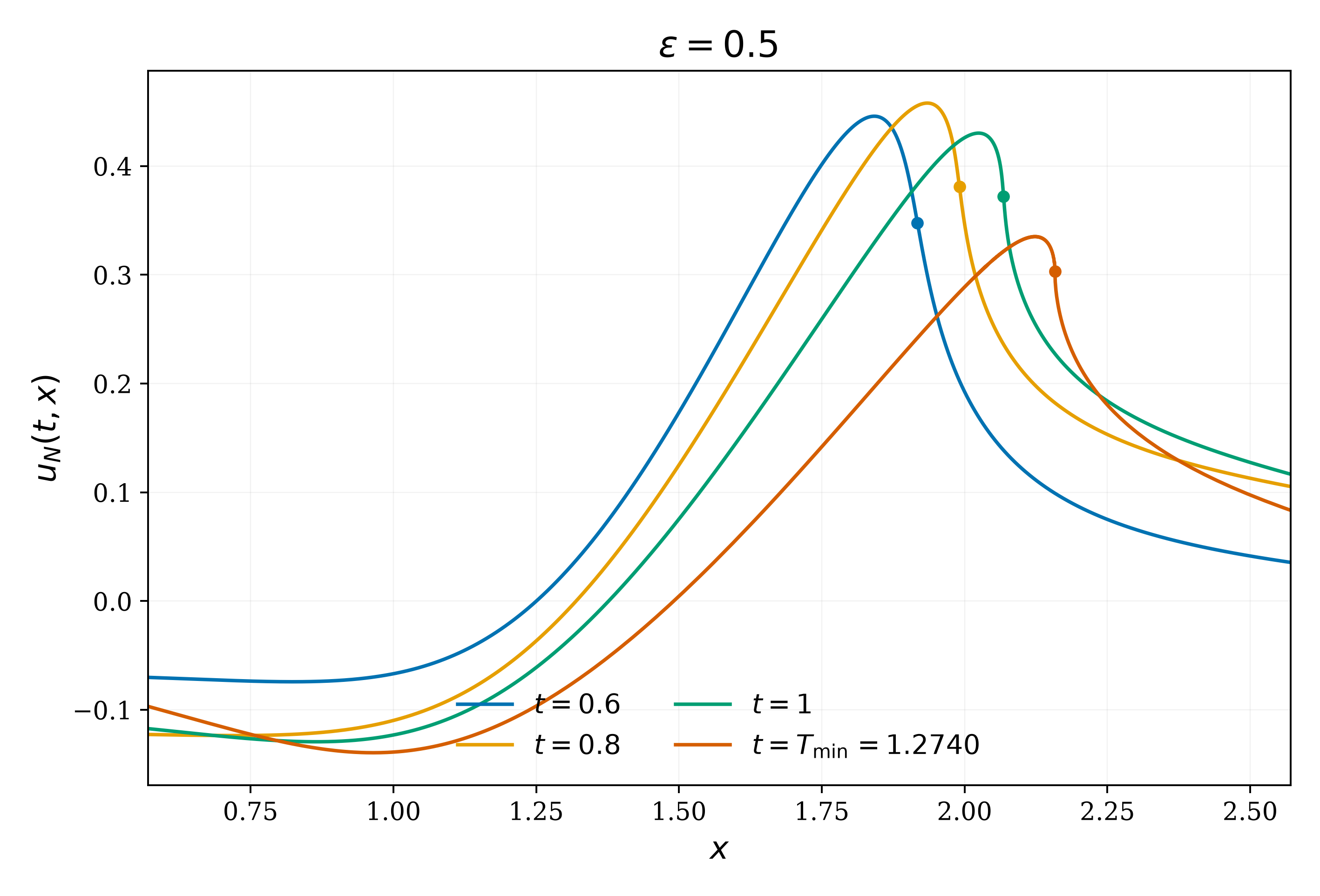}
\includegraphics[width=0.45\linewidth]{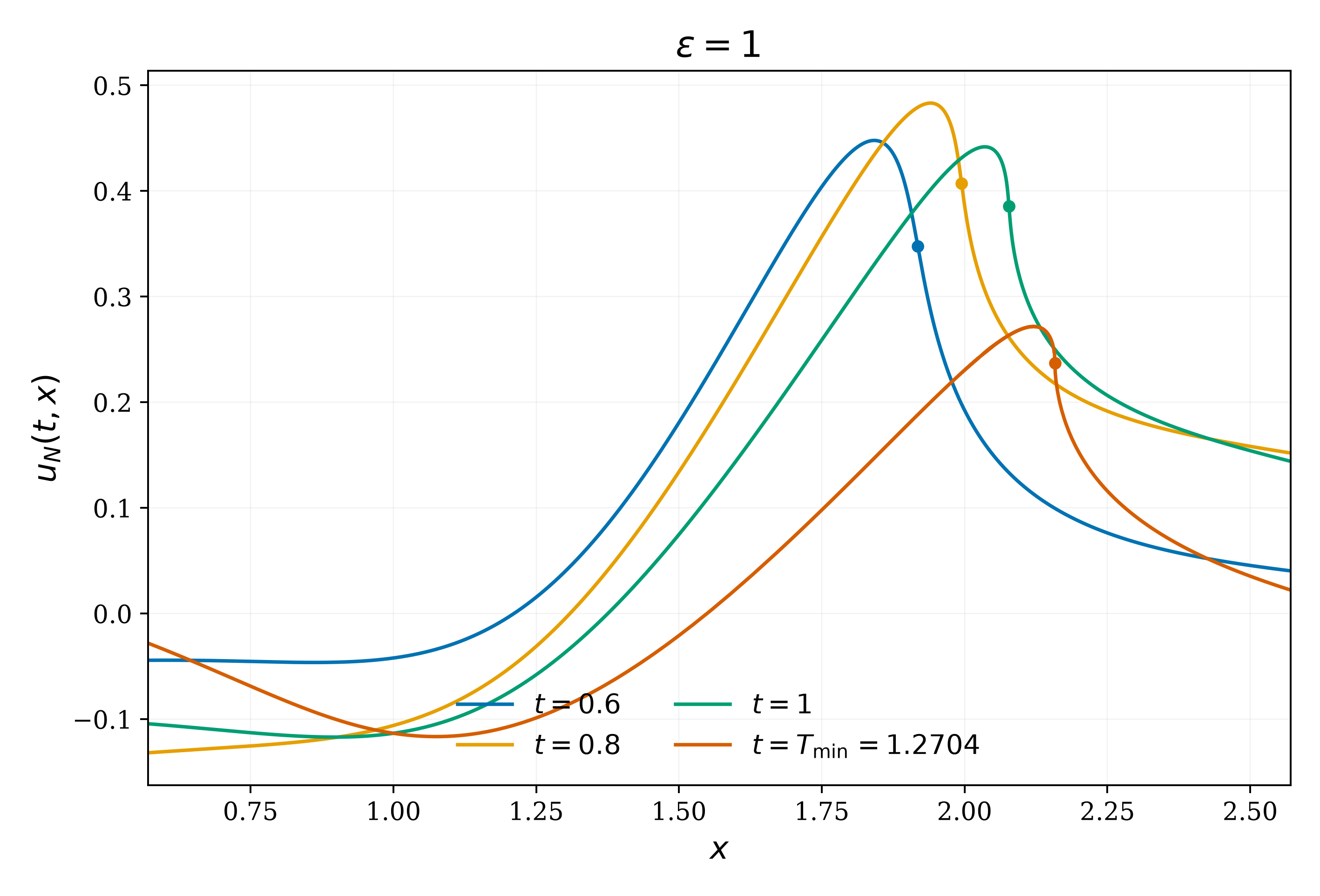} \caption{Pathwise numerical profiles for \(\varepsilon=0.01\), \(0.1\), \(0.5\), and \(1\). 
% In each panel, the profiles at \(t=0.6\), \(0.8\), \(1\), and \(T_{\min}\) are superposed in blue, orange, bluish green, and vermilion, respectively. The colored circular markers identify the points of most negative numerical slope within \(\pi/2-1\le x\le\pi/2+1\).
} \label{fig:num-wave-four-snapshot-profiles} \end{figure} 
\begin{figure}[!htbp] \centering \includegraphics[width=0.47\linewidth]{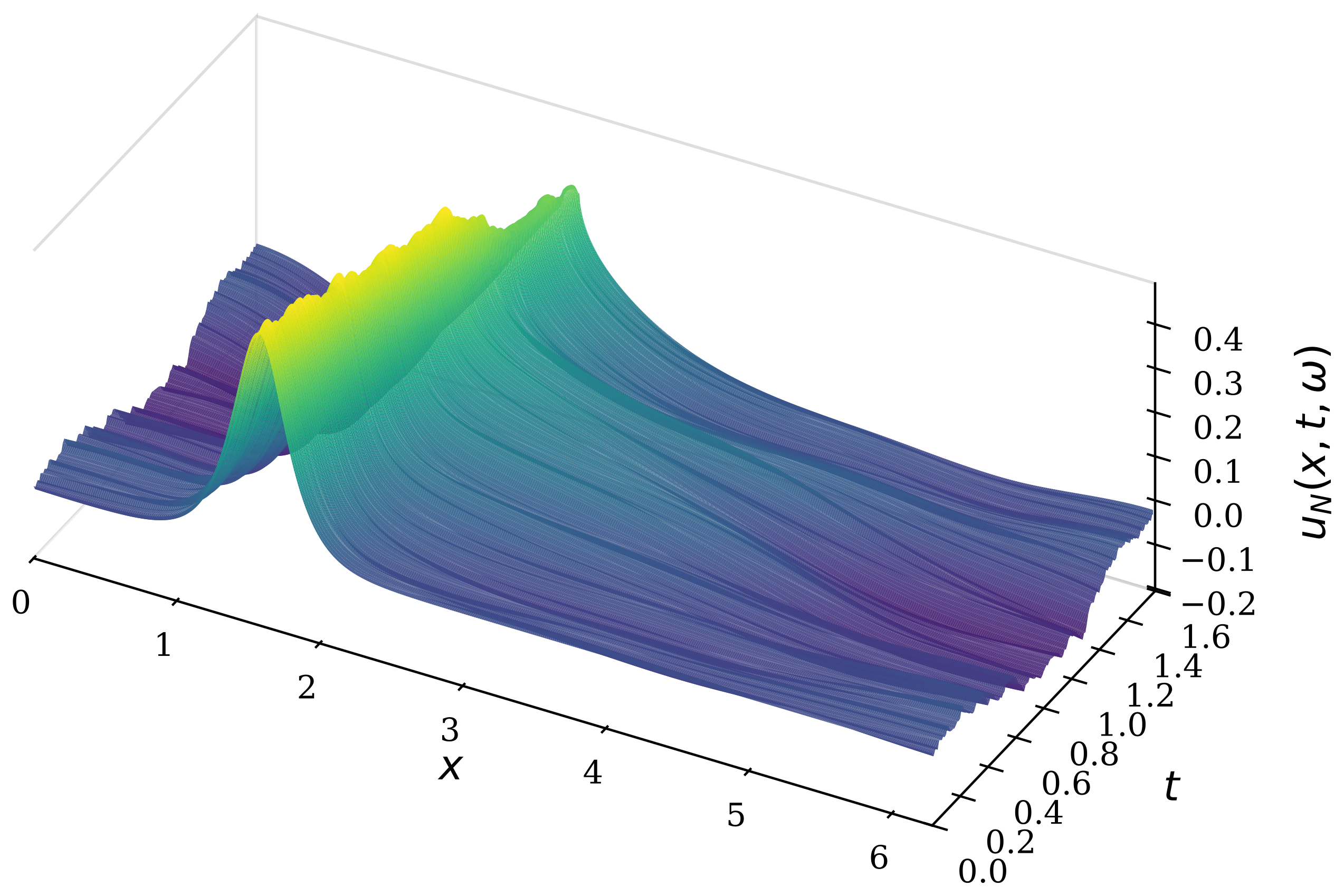}\includegraphics[width=0.47\linewidth]{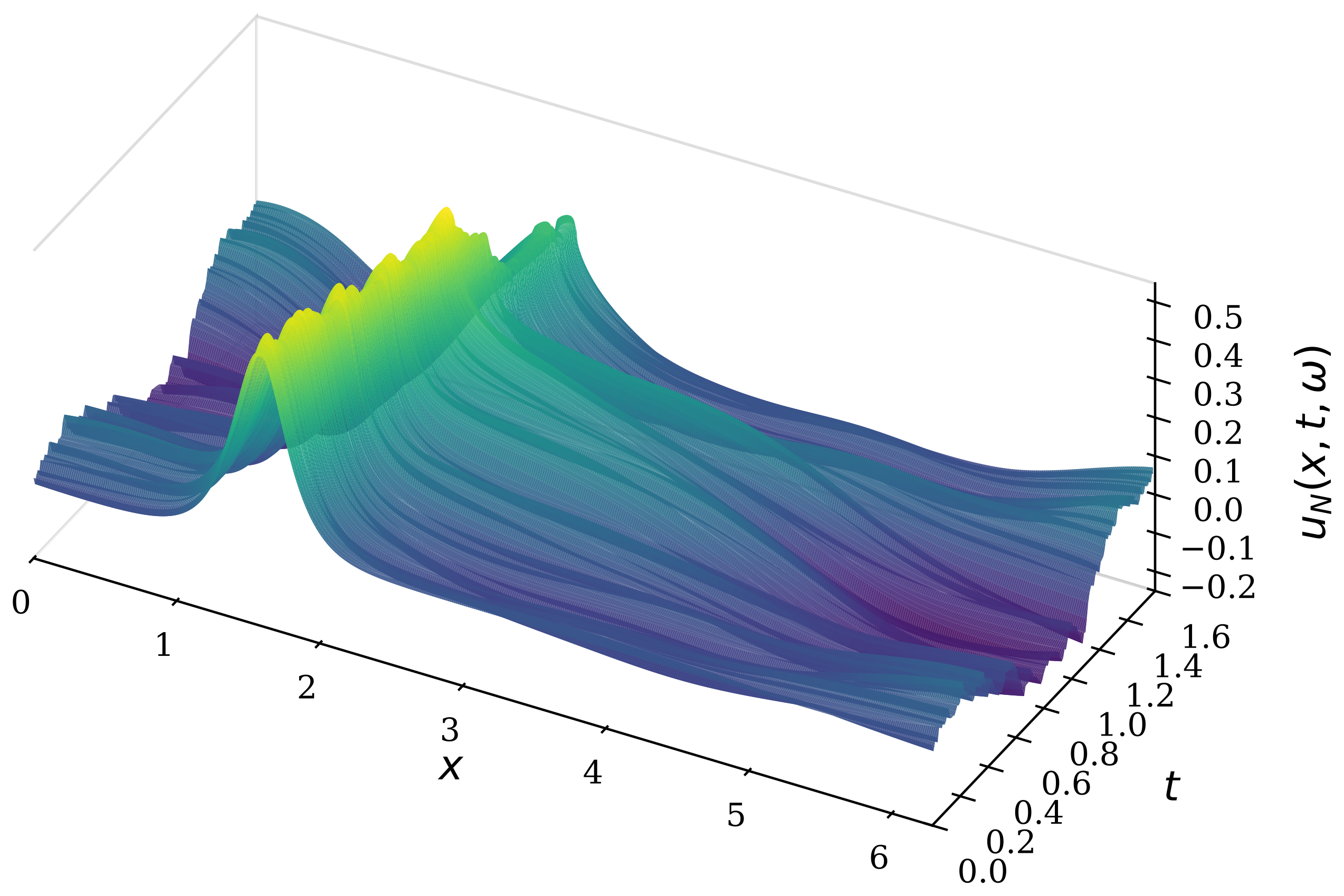} \caption{Pathwise space--time surfaces for \(\varepsilon=0.5\) (left) and \(\varepsilon=1\) (right).} \label{fig:num-wave-surfaces} \end{figure}

% \subsection{Numerical wave breaking signatures and slope amplification}
% We finally examine whether individual numerical trajectories exhibit the
% one-sided slope amplification associated with wave breaking. The analytical
% criterion established above has the form
% \begin{equation}
% \label{eq:num-wave-breaking-criterion}
%     \sup_{0\le t<\tau_\infty(\omega)}
%     \sup_{x\in\mathbb T_L}u_x(t,x,\omega)<\infty,
%     \qquad
%     \lim_{t\uparrow\tau_\infty(\omega)}
%     \inf_{x\in\mathbb T_L}u_x(t,x,\omega)=-\infty .
% \end{equation}
% {\color{red}Since a finite-dimensional Fourier approximation remains smooth and
% cannot develop an infinite slope, the following computations should be
% interpreted only as pathwise numerical evidence supporting the
% wave breaking mechanism in \eqref{eq:num-wave-breaking-criterion}, rather
% than as a numerical proof of wave breaking.}
\iffalse
A finite-dimensional Fourier approximation cannot attain an infinite slope.
The computations below therefore provide pathwise numerical evidence
consistent with \eqref{eq:num-wave-breaking-criterion}; they do not constitute
a numerical proof of blow-up.
\fi

To illustrate the pathwise steepening mechanism associated with the discrete
Riccati-type analysis in Theorem~\ref{thm:fully_discrete_riccati}, we
use the localized zero-mean pulse introduced above.
%, whose pronounced descending flank is chosen to favor early wave breaking.
On $\mathbb T_{2\pi}$, represented
by $[0,2\pi)$, we initialize the numerical solution by
$U_N^0(x_j)
=\frac12\operatorname{sech}((x_j-\pi/2)/(1/6))-c_{0,N_{\mathrm F}}$,
where $x_j=2\pi j/N_{\mathrm F}$ and
$c_{0,N_{\mathrm F}}
=N_{\mathrm F}^{-1}\sum\limits_{j=0}^{N_{\mathrm F}-1}
\frac12\operatorname{sech}((x_j-\pi/2)/(1/6)).$
We use $N=5000$, corresponding to $N_{\mathrm F}=10001$ grid points,
with $\tau=4\times10^{-4}$ and $T=1.6$.
%The noise truncation is the same as in the preceding experiments, with $50$ wavenumbers in each of the sine and cosine components.
%The simulations for and $1$ use the same realizations of all driving Brownian motions. 
Figure~\ref{fig:num-wave-four-snapshot-profiles} shows the numerical
profiles at $t=0.6,0.8,1$, and $T_{\min}$ for each noise intensity $\varepsilon=0.01,0.1,0.5$.
Here,
$m_N(t_k)
=
\min\limits_{0\leq j<N_{\mathrm F}}
\partial_x U_N^k(x_j), t_k=k\tau,$
is the minimum slope sampled on the spatial grid, and $T_{\min}$
is the earliest grid time at which $m_N(t_k)$ attains its minimum
over $0\leq t_k\leq T$. %{\color{red} Thus, $T_{\min}$ identifies the most negative sampled slope within the simulation window; it is not identified with the theoretical wave-breaking time.} 
The displayed profiles show pronounced sharpening of the descending
front without a comparable increase in wave amplitude.
For the larger noise intensities, the profiles exhibit more visible
fluctuations, while localized steepening remains apparent.
The space--time surfaces in
Figure~\ref{fig:num-wave-surfaces} for $\varepsilon=0.5$ and $1$
provide a complementary view of this evolution.
Together, these observations provide qualitative numerical evidence
of negative-slope amplification consistent with the discrete
Riccati-type mechanism, rather than a numerical verification of
finite-time slope blow-up.

\section*{Acknowledgments}
This work is supported by MOST National Key R\&D Program No. 2024FA1015900, the National Natural Science Foundation of China (No. 12471386 and
No. 12671464). 

% The bibliography is generated only from citations in the active text.
% The former broad \nocite{...} list is intentionally disabled so that
% uncited background references do not appear in the final bibliography.

\bibliographystyle{plain}
\bibliography{references}
\end{document}